\documentclass[]{amsart}
\usepackage[utf8]{inputenc}
\usepackage{amsmath}
\usepackage{amsfonts}
\usepackage{amsthm}
\usepackage{enumerate}
\usepackage{amssymb}
\usepackage{color}
\usepackage{todonotes}
\usepackage{comment}
\usepackage{tikz}
\usepackage{float}
\usepackage{todonotes}
\usetikzlibrary{decorations.markings}
\usetikzlibrary{arrows}
\usepackage{graphicx}
\usepackage{caption}
\usepackage{subcaption}
\usepackage{xfrac}
\usepackage{mathtools}
\usepackage[all]{xy}

\newcommand{\ts}{\textsuperscript}

\newcommand{\Z}{\mathbb{Z}}

\newcommand{\N}{\mathbb{N}}

\newcommand{\G}{H_n\rtimes S_n}

\renewcommand{\Vec}[1]{\underline{\mathbf{#1}}}

\def\quotient#1#2{%
   \lower0.25ex\hbox{$#2$}\big\backslash\raise0.25ex\hbox{$#1$}%
}

\renewcommand{\le}{\leqslant}
\renewcommand{\ge}{\geqslant}

\theoremstyle{plain}% default
\newtheorem{thm}{Theorem}[section]
\newtheorem{lem}[thm]{Lemma}
\newtheorem{prop}[thm]{Proposition}

\newtheorem{nota}[thm]{Notation}
\newtheorem{rem}[thm]{Remark}
\newtheorem*{remR}{Remark}

\newtheorem{thmR}{Theorem}

\makeatletter
\newtheorem*{rep@theorem}{\rep@title}
\newcommand{\newreptheorem}[2]{%
\newenvironment{rep#1}[1]{%
 \def\rep@title{#2 \ref{##1}}%
 \begin{rep@theorem}}%
 {\end{rep@theorem}}}
\makeatother
\newreptheorem{theorem}{Theorem}

\newtheorem*{thm*}{Theorem}
\newtheorem*{lem*}{Lemma}
\newtheorem*{prop*}{Proposition}
\newtheorem*{cor*}{Corollary}
\newtheorem*{qu*}{Question}
\newtheorem*{dt*}{Definition and Theorem}
\newtheorem*{not*}{Notation}
\newtheorem*{defn*}{Definition}
\newtheorem*{exmp*}{Example}
\newtheorem*{exmps*}{Examples}
\newtheorem*{dprop*}{Definition and Proposition}
\newtheorem*{fact*}{Fact}

\theoremstyle{definition}
\newtheorem{defn}[thm]{Definition}

\DeclareMathOperator\Supp{supp}
\DeclareMathOperator\FSym{FSym}
\DeclareMathOperator\FAlt{FAlt}
\DeclareMathOperator\Sym{Sym}
\DeclareMathOperator\Aut{Aut}

\DeclareMathOperator\Inn{Inn}

\DeclareMathOperator\FP{FP}

\DeclareMathOperator\id{id}

\begin{document}
\title{Twisted Conjugacy in Houghton's groups}
\author{Charles Garnet Cox}
\address{Mathematical Sciences, University of Southampton, SO17 1BJ, UK}
\email{cpgcox@gmail.com}
\thanks{}

\subjclass[2010]{Primary: 20F10, 20B99}

\keywords{uniform twisted conjugacy problem, Houghton group, permutation group, orbit decidability, conjugacy problem for commensurable groups, computable centralisers.}

\date{July $4$, 2017}

\dedicatory{}
\begin{abstract} For a fixed $n\in \{2, 3, \ldots\}$, the Houghton group $H_n$ consists of bijections of $X_n=\{1,\ldots,n\} \times \N$ that are `eventually translations' of each copy of $\N$. The Houghton groups have been shown to have solvable conjugacy problem. In general, solvable conjugacy problem does not imply that all finite extensions and finite index subgroups have solvable conjugacy problem. Our main theorem is that a stronger result holds: for any $n\in \{2, 3, \ldots\}$ and any group $G$ commensurable to $H_n$, $G$ has solvable conjugacy problem.
\end{abstract}
\maketitle
\section{Introduction}
Given a presentation $\langle S\mid R\rangle=G$, a `word' in $G$ is an ordered $f$-tuple $a_1^{}\ldots a_f$ with $f \in \N$ and each $a_i \in S\cup S^{-1}$. Dehn's problems and their generalisations (known as decision problems) ask seemingly straightforward questions about finite presentations. The problems that we shall consider include:
\begin{itemize}
\item WP($G$), the \emph{word problem} for $G$: show there exists an algorithm which given any two words $a,b \in G$,  decides whether $a=_{G}b$ or $a\ne_{G}b$ i.e.\ whether these words represent the same element of the group. This is equivalent to asking whether or not $ab^{-1}=_{G}1$. There exist finitely presented groups where this problem is undecidable (see \cite{Novikov} or \cite{Boone}).
\item CP($G$), the \emph{conjugacy problem} for $G$: show there exists an algorithm which given any two words $a,b \in G$, decides whether or not there exists an $x \in G$ such that $x^{-1}ax=b$. Since $\{1_G\}$ is always a conjugacy class of $G$, if CP($G$) is solvable then so is WP($G$). There exist groups where WP($G$) is solvable but CP($G$) is not (e.g.\ see \cite{classification}) and so CP($G$) is strictly weaker than WP($G$).
\item TCP$_{\phi}$($G$), the $\phi$-\emph{twisted conjugacy problem} for $G$: for a fixed $\phi \in \Aut(G)$, show there exists an algorithm which given any two words $a,b \in G$, decides whether or not there exists an $x \in G$ such that $(x^{-1})\phi ax=b$ (meaning that $a$ is $\phi$-twisted conjugate to $b$). Note that TCP$_{Id}$($G$) is CP($G$). 
\item TCP($G$), the \emph{(uniform) twisted conjugacy problem} for $G$: show there exists an algorithm which given any $\phi \in \Aut(G)$ and any two words $a, b \in G$, decides whether or not they are $\phi$-twisted conjugate. There exist groups $G$ such that CP($G$) is solvable but TCP($G$) is not (e.g.\ see \cite{orbitdecide}).
\end{itemize}

Should any of these problems be solved for one finite presentation, then they may be solved for any other finite presentation of that group. We therefore say that such problems are solvable if an algorithm exists for one such presentation. Many decision problems may also be considered for any group that is recursively presented. 

We say that $G$ is a finite extension of $H$ if $H\unlhd G$ and $H$ is finite index in $G$. If CP($G$) is solvable, then we do not have that finite index subgroups of $G$ or finite extensions of $G$ have solvable conjugacy problem, even if these are of degree 2. Explicit examples can be found for both cases (see \cite{conjsub} or \cite{Gorjaga1975}). Thus it is natural to ask, if CP($G$) is solvable, whether the conjugacy problem holds for finite extensions and finite index subgroups of $G$. The groups we investigate in this paper are Houghton groups (denoted $H_n$ with $n \in \N$).
\begin{thmR} Let $n\in \{2, 3, \ldots\}$. Then TCP($H_n$) is solvable.\label{main2}
\end{thmR}

We say that two groups $A$ and $B$ are commensurable if there exists $N_A\cong N_B$ where $N_A$ is normal and finite index in $A$ and $N_B$ is normal and finite index in $B$. 

\begin{thmR} Let $n\in \{2, 3, \ldots\}$. Then, for any group $G$ commensurable to $H_n$, CP($G$) is solvable.\label{main5}
\end{thmR}

As a consequence of this work we also obtain.
\begin{thmR}\label{main6} There exists an algorithm which takes as input any $n\in \{2, 3, \ldots\}$ and any $a\in \G$, and which decides whether or not $C_{H_n}(a)$ is finitely generated. If $C_{H_n}(a)$ is finitely generated, then the algorithm also outputs a finite generating set for $C_{H_n}(a)$.
\end{thmR}
We structure the paper as follows. In Section 2 we introduce the Houghton groups, make some simple observations for them, and reduce TCP($H_n$) to a problem similar to CP($\G$), the difference being that, given $a, b \in \G$, we are searching for a conjugator $x \in H_n$. This occurs since, for all $n\in \{2, 3, \ldots\}$, $\Aut(H_n)\cong\G$. In Section 3 we describe the orbits of elements of $\G$ and produce identities that a conjugator of elements in $\G$ must satisfy. These are then used in Section 4 to reduce our problem of finding a conjugator in $H_n$ to finding a conjugator in the subgroup of $H_n$ consisting of all finite permutations (which we denote by $\FSym$, see Notation \ref{notnFSym} below). Constructing such an algorithm provides us with Theorem \ref{main2}. In Section 5 we use Theorem \ref{main2} and \cite[Thm. 3.1]{orbitdecide} to prove our main result, Theorem \ref{main5}. In Section 6 we discuss the structure of $C_{H_n}(a)$ where $a\in \G$ and prove Theorem \ref{main6}.

\vspace{2mm}
\noindent\textbf{Acknowledgements.} I thank the authors of \cite{ConjHou} whose work is drawn upon extensively. I especially thank the author Armando Martino, my supervisor, for his encouragement and the many helpful discussions which have made this work possible. I thank Peter Kropholler for his suggested extension which developed into Theorem \ref{main5}. Finally, I thank the referee for their time and many helpful comments.

\section{Background}\label{sectionbackground}
As with the authors of \cite{ConjHou}, the author does not know of a class that contains the Houghton groups and for which the conjugacy problem has been solved.
\subsection{Houghton's groups}\label{sectionhn} Throughout we shall consider $\N:=\{1, 2, 3, \ldots\}$. For convenience, let $\Z_n:=\{1,\ldots,n\}$. For a fixed $n \in \N$, let $X_n:=\Z_n \times \N$. Arrange these $n$ copies of $\N$ as in Figure \ref{picturehoupic} below (so that the $k\ts{th}$ point from each copy of $\N$ form the vertices of a regular $n$-gon). For any $i \in \Z_n$, we will refer to the set $i \times \N$ as a \emph{branch} or \emph{ray} and will let $(i,m)$ denote the $m\ts{th}$ point on the $i\ts{th}$ branch.

\begin{nota}\label{notnFSym} For a non-empty set $X$, the set of all permutations of $X$ form a group which we denote $\Sym(X)$. Those permutations which have finite support (i.e.\ move finitely many points) form a normal subgroup which we will denote $\FSym(X)$. If there is no ambiguity for $X$, then we will write just $\Sym$ or $\FSym$ respectively.
\end{nota}

Note that, if $X$ is countably infinite, then $\FSym(X)$ is countably infinite but is not finitely generated, and $\Sym(X)$ is uncountable and so uncountably generated.
\begin{defn}\label{definitionhn} Let $n \in \N$. The $n\ts{th}$ \emph{Houghton group}, denoted $H_n$, is a subgroup of $\Sym(X_n)$. An element $g \in \Sym(X_n)$ is in $H_n$ if and only if there exist constants $z_1(g),\ldots,z_n(g)\in\N$ and $(t_1(g),\ldots,t_n(g))\in\Z^n$ such that, for all $i \in \Z_n$,
\begin{align}
 &(i,m)g = (i, m + t_i(g))\; \textrm{for all} \;m\ge z_i(g).\label{houcond}
\end{align}
\end{defn}
For simplicity, the numbers $z_1(g),\ldots, z_n(g)$ are assumed to be minimal i.e.\ for any $m'<z_i(g)$, $(i,m')g\ne (i, m'+t_i(g))$. The vector $t(g):=(t_1(g),\ldots,t_n(g))$ represents the `eventual translation length' for each $g$ in $H_n$ since $t_i(g)$ specifies how far $g$ moves the points $\{(i,m)\mid m\ge z_i(g)\}$. We shall say that these points are those which are `far out', since they are the points where $g$ acts in the regular way described in (\ref{houcond}). As $g$ induces a bijection from $X_n$ to $X_n$, we have that
\begin{align}
 &\sum\limits_{i=1}^{n}t_i(g)=0.\label{houcond2}
\end{align}
Given $g \in H_n$, the numbers $z_i(g)$ may be arbitrarily large. Thus $\FSym(X_n)\le H_n$.  Also, for any $n\in \{2, 3, \ldots\}$, we have (as proved in \cite{ses}) the short exact sequence 
 \begin{align*}
  1\longrightarrow \FSym(X_n)\stackrel{}{\longrightarrow} H_n\stackrel{}{\longrightarrow} \Z^{n-1}\longrightarrow 1
 \end{align*}
where the homomorphism $H_n \rightarrow \Z^{n-1}$ is given by $g\mapsto (t_1(g),\ldots,t_{n-1}(g))$.

These groups were introduced in \cite{Houghton} for $n\in \{3, 4, \ldots\}$. The generating set that we will use when $n\in \{3, 4, \ldots\}$ is $\{g_i\mid i=2,\ldots, n\}$ where, for each $i\in \{2,\ldots, n\}$ and $(j,m)\in X_n$,
\begin{equation*}
 (j,m)g_i = \left\{\begin{array}{ll}(1, m+1) & \mathrm{if}\ j=1\\ (1,1) & \mathrm{if}\ j=i, m=1\\ (i, m-1) & \mathrm{if}\ j=i, m>1\\ (j,m) & \mathrm{otherwise.}\end{array}\right.
\end{equation*}
Note, for each $i$, $t_1(g_i)=1$ and $t_j(g_i)=-\delta_{i,j}$ for $j \in \{2,\ldots,n\}$. Figure \ref{picturehoupic} shows a geometric visualisation of $g_2, g_4 \in H_5$. Throughout we shall take the vertical ray as the `first ray' (the set of points $\{(1, m)\mid m \in \N\}$) and order the other rays clockwise.

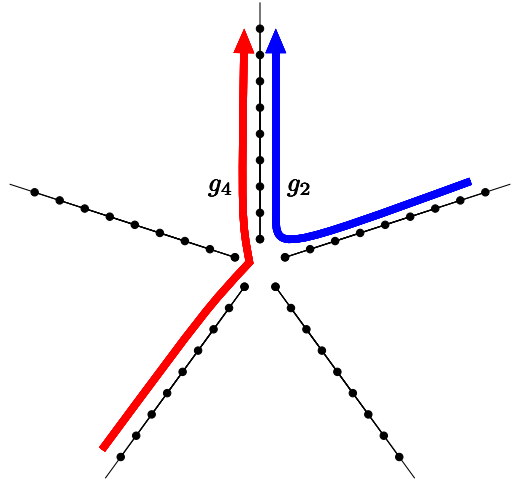
\begin{figure}[ht] 
\begin{tikzpicture}[scale=0.7]
\foreach \x in {18,90,...,306}
{
\draw (\x:5cm)  -- (\x:0.5cm);
\filldraw (\x:4.5cm) circle (2pt)-- (\x:0.5cm) circle (2pt);
\filldraw (\x:4cm) circle (2pt)-- (\x:1cm) circle (2pt);
\filldraw (\x:3.5cm) circle (2pt)-- (\x:1.5cm) circle (2pt);
\filldraw (\x:3cm) circle (2pt)-- (\x:2.5cm) circle (2pt);
\filldraw (\x:2cm) circle (2pt)-- (\x:0.5cm) circle (2pt);

\draw  [
    -triangle 45,
    line width=0.4mm,
    color=red,
    postaction={draw=red, line width=0.1cm, shorten >=0.3cm, -}
]  (-3, -3.5) .. controls (-1,-0.8)  .. (-0.2,0.06) .. controls +(-0.15,0.8) .. (-0.3,4.5);
\draw [
    -triangle 45,
    line width=0.4mm,
    color=blue,
    postaction={draw=blue, line width=0.1cm, shorten >=0.3cm, -}
]  (4, 1.6) .. controls (0.3,0.25) ..  (0.3,1) .. controls+(0,1.2) .. (0.3,4.5);
\node [below] at (-0.75,1.8) {$g_4$};
\node [below] at (0.75,1.8) {$g_2$};
}
\end{tikzpicture}
\caption{Part of the set $X_5$ and a geometrical representation of the action of the standard generators $g_2^{}, g_4^{} \in H_5$.}\label{picturehoupic}\end{figure}

We shall now see that, for any $n\in \{3, 4, \ldots\}$, the set $\{g_i\mid i=2,\ldots,n\}$ generates $H_n$. First, any valid eventual translation lengths (those satisfying (\ref{houcond2})) can be obtained by these generators. Secondly, the commutator (which we define as $[g,h]:=g^{-1}h^{-1}gh$ for every $g,h$ in $G$) of any two distinct elements $g_i, g_j \in H_n$ is a 2-cycle, and so conjugation of this 2-cycle by some combination of $g_k$'s will produce a 2-cycle with support equal to any two points of $X_n$. This is enough to produce any element that is `eventually a translation' i.e.\ one that satisfies condition (\ref{houcond}), and so is enough to generate all of $H_n$. An explicit finite presentation for $H_3$ can be found in \cite{johnson}, and this was generalised in \cite{sanglee} by providing finite presentations for $H_n$ for all $n>3$. 

We now describe $H_1$ and $H_2$. If $g \in H_1$, then $t_1(g)=0$ (since the eventual translation lengths of $g$ must sum to 0 by condition (\ref{houcond2})) and so $H_1= \FSym(\N)$. For $H_2$ we have $\langle g_2\rangle\cong\Z$. Using a conjugation argument similar to the one above, it can be seen that a suitable generating set for $H_2$ is  $\{g_2, ((1,1)\;(2,1))\}$. These definitions of $H_1$ and $H_2$ agree with the result for $H_n$ in \cite{brown}, that (for $n\ge3$) each $H_n$ is $\FP_{n-1}$ but not $\FP_n$ i.e.\ $H_1$ is not finitely generated and $H_2$ is finitely generated but not finitely presented. Since $H_1 = \FSym(\N)$ and $\Aut(\FSym(\N)) \cong \Sym(\N)$ (see, for example, \cite{Permutationgroups} or \cite{grouptheory87}) we will work with $H_n$ where $n\in \{2, 3, \ldots\}$.
\begin{comment}
 \subsection{The structure of $\Aut(H_n)$}\label{sectionstructre}
see 'resubmit.tex' for this section
\end{comment}

\subsection{A reformulation of TCP($H_n$)}\label{Sectionequivproblem} 
We require knowledge of $\Aut(H_n)$. We noted above that $\Aut(H_1)\cong \Sym(\N)$, and so will work with $n\in \{2, 3, \ldots\}$.
\begin{not*} Let $g \in \Sym(X)$. Then $(h)\phi_g:=g^{-1}hg$ for all $h \in G$.
\end{not*}
From \cite{Hou2}, we have for all $n\in \{2, 3, \ldots\}$ that $N_{\Sym(X_n)}(H_n)\cong\Aut(H_n)$ via the map $\rho\mapsto \phi_\rho$ and that $N_{\Sym(X_n)}(H_n)\cong\G$. We will make an abuse of notation and consider $\G$ as acting on $X_n$ via the natural isomorphism $N_{\Sym(X_n)}(H_n)\le \Sym(X_n)$. Here $\Inn(H_n)\cong H_n$ because $H_n$ is centreless, and $S_n$ acts on $X_n$ by isometrically permuting the rays, where $g \in \G$ is an isometric permutation of the rays if and only if there exists a $\sigma \in S_n$ such that $(i,m)g=(i\sigma, m)$ for all $m \in \N$ and all $i \in \Z_n$.

\begin{not*} For any given $g \in \G$, let $\Psi: \G\rightarrow S_n$, $g\mapsto \sigma_g$ where $\sigma_g$ denotes the isometric permutation of the rays induced by $g$. Furthermore, let $\omega_g:= g\sigma_g^{-1}$. Thus, for any $g \in \G$, we have $g=\omega_g\sigma_g$ and will consider $\omega_g \in H_n$ and $\sigma_g \in S_n$. We shall therefore consider any element $g$ of $\G$ as a permutation of $X_n$ which is eventually a translation (denoted $\omega_g$) followed by an isometric permutation of the rays $\sigma_g$.
\end{not*}

\begin{defn} Let $G\le \G$ and $g, h \in \G$. Then we shall say $g$ and $h$ are $G$-conjugated if there is an $x \in G$ such that $x^{-1}gx=h$.
\end{defn}

We now relate twisted conjugacy in $H_n$ to conjugacy in $\G$. Let $c \in \G$. Then the equation for $\phi_c$-twisted conjugacy becomes:
\begin{align*}
(x^{-1})\phi_c g x=h
\Leftrightarrow c^{-1}x^{-1}cgx=h
\Leftrightarrow x^{-1}(\omega_c\sigma_cg)x=\omega_c\sigma_ch.
\end{align*}
Thus, for any $n\in \{2, 3, \ldots\}$, two elements $g, h \in H_n$ are $\phi_c$-twisted conjugate if and only if $\omega_c\sigma_cg$ and $\omega_c\sigma_ch \in \G$ are $H_n$-conjugated. Note that, if $g, h \in \G$ are $H_n$-conjugated, then $\sigma_g=\sigma_h$. Thus for the remainder of this note $a$ and $b$ will refer to the elements in $\G$ that we wish to decide are $H_n$-conjugated, where $\sigma_a=\sigma_b$. In order to solve TCP($H_n$), we will therefore produce an algorithm to search for an $x \in H_n$ which conjugates $a$ to $b$.

\section{Computations in $\G$}\label{sectioncomputations}
Our first lemma provides the generating set that we will use for $\G$.
\begin{lem}\label{lemmagenset} Let $n\in \{2, 3, \ldots\}$. Then $\G$ can be generated by 3 elements. 
\end{lem}
\begin{proof}
If $n\ge3$, two elements can be used to generate all of the isometric permutations of the rays. Our third generator will be $g_2$, the standard generator for $H_n$. Conjugating $g_2$ by the appropriate isometric permutations of the rays produces the set $\{g_i\mid i=2,\ldots,n\}$, which can then be used to generate all permutations in $H_n$. For $H_2\rtimes S_2$ we note that $H_2$ is 2-generated and that $S_2$ is cyclic.
\end{proof}

\subsection{The orbits of elements in $\G$}\label{sectionorbits}
Our main aim for this section is to describe the orbits of any element $g \in \G$ `far out'. For elements of $H_n$, any element eventually acted like a translation. In a similar way, any element of $\G$ eventually moves points in a uniform manner. More specifically, $g \in \Sym(X_n)$ is in $\G$ if and only if there exist constants $z_1(g),\ldots,z_n(g) \in \N$, $(t_1(g), \ldots, t_n(g)) \in \Z^n$, and a permutation $\sigma \in S_n$ such that for all $i \in \Z_n$
\begin{align}\label{houextcondsimpler}
(i,m)g=(i\sigma, m+t_i(g))\;\text{for all}\;m\ge z_i(g).
\end{align}
If $g \in \G$, then $g=\omega_g\sigma_g$. Therefore for any $g \in \G$, $\sigma_g$ (the isometric permutation of the rays induced by $g$) will induce the permutation denoted $\sigma$ in (\ref{houextcondsimpler}), we have that $t_i(\omega_g)=t_i(g)$ for all $i \in \Z_n$, and $z_1(\omega_g), \ldots, z_n(\omega_g)$ are suitable values for the constants $z_1(g), \ldots, z_n(g)$.
\begin{defn}\label{classes} Let $g \in \G$ and $i \in \Z_n$. Then a \emph{class} of $\sigma_g$, denoted $[i]_{g}$, is the support of the disjoint cycle of $\sigma_g$ which contains $i$ i.e.\ $[i]_g=\{i\sigma_g^d\mid d \in \Z\}$. Additionally, we define the \emph{size} of a class $[i]_g$ to be the length of the cycle of $\sigma_g$ containing $i$, i.e. the cardinality of the set $[i]_g$, and denote this by $|[i]_g|$.
\end{defn}
We shall choose $z_1(g), \ldots, z_n(g) \in \N$ to be the smallest numbers such that
\begin{align}\label{houextcond}
 (i,m)g^d=(i\sigma_g^d, m+\sum\limits_{s=0}^{d-1}t_{i\sigma_g^s}(g))\;\textrm{for all}\;m\ge z_i(g)\;\textrm{and all}\;1\le d\le |[i]_g|.
\end{align}
Note that, for any $g \in \G$ and all $i \in \Z_n$, we have $z_i(g)\ge z_i(\omega_g)$. We now justify the introduction of condition (\ref{houextcond}). Consider a $g \in \G$, $i \in \Z_n$, and $m \in \N$ such that
\[z_i(\omega_g)\le m<z_i(g)\;\text{and}\;m+t_{i}(\omega_g)<z_{i\sigma_g}(\omega_g).\]
This would mean that $(i, m)g=(i\sigma_g, m+t_{i}(\omega_g))$, but it may also be that
\[(i, m)g^2=(i\sigma_g, m+t_{i}(g))g\ne (i\sigma_g^2, m+t_{i}(g)+t_{i\sigma_g}(g)).\]
Thus the condition (\ref{houextcond}) above means that, for any $g \in \G$, the numbers $z_1(g), \ldots, z_n(g)$ capture the `eventual' way that $g$ permutes the points of a ray.

Let us fix some $g \in \G$. Consider if $\sigma_g$ acts trivially on a particular branch $i'$. This will mean that this branch has orbits like those occurring for elements of $H_n$. If $t_{i'}(g)=0$, then $g$ leaves all but a finite number of points on this branch fixed. If $t_{i'}(g)\ne0$, then for any given $m'>z_{i'}(g)$,
\[\{(i', m')g^d\mid d \in \Z\}\supset\{(i', m)\mid m\ge m'\;\text{and}\;m\equiv m' \bmod|t_{i'}(g)|\}.\]

Notice that, for any $g \in \G$, the $\sigma_g$-classes form a partition of $\Z_n$, relating to the branches of $X_n$. We now consider the case $|[k]_g|>1$. We first note, for any $i \in [k]_g$ and $m_1\ge z_{i}(g)$, that
\begin{align}\label{eqnsamebranch}
(i,m_1)g^{|[k]_g|}\in\{(i,m)\mid m \in \N\}
\end{align}
and that for any $1\le p< |[k]_g|$, $i \in [k]_g$, and $m_1\ge z_{i}(g)$,
\[(i,m_1)g^p \not\in \{(i, m)\mid m \in \N\}.\]
In fact we may compute $(i,m)g^{|[k]_g|}$ for any $i \in [k]_g$ and $m\ge z_i(g)$,
\begin{align}\label{eqnimageofg}
(i,m)g^{|[k]_g|}=(i, m+\sum\limits_{s=0}^{\mathclap{|[k]_g|-1}}t_{i\sigma_g^s}(g)).
\end{align}
In light of this, we introduce the following.
\begin{defn}
For any $g \in \G$, class $[i]_{g}$, and $k \in [i]_{g}=\{i_1, i_2, \ldots, i_q\}$, let
 \begin{align*}
{t}_{[k]}(g):= \sum\limits_{s=1}^{q} t_{i_s}(g).
\end{align*}
\end{defn}
Hence, if $t_{[k]}(g)=0$, then for all $i' \in [k]_g$ and $m'\ge z_{i'}(g)$, the point $(i',m')$ will lie on an orbit of length $|[k]_g|$. If $t_{[k]}(g)\ne0$, then (\ref{eqnimageofg}) states that for all $i' \in [k]_g$ and all $m'\ge z_{i'}(g)$, that $(i',m')g^{|[k]_g|}\ne(i',m')$. Hence when $t_{[k]}(g)=0$, almost all points of the $k$\ts{th} ray will lie on an orbit of $g$ of length $|[k]_g|$, and when $t_{[k]}(g)\ne0$ almost all points of the $k$\ts{th} ray will lie in an infinite orbit of $g$. Since different arguments will be required for finite orbits and infinite orbits, we introduce the following notation.
\begin{not*} Let $g \in \G$. Then $I(g):=\{i\in \Z_n\mid t_{[i]}(g)\ne0\}$ consists of all $i \in \Z_n$ corresponding to rays of $X_n$ which have infinite intersection with some infinite orbit of $g$. Let $I^c(g):=\Z_n\setminus I(g)$, the complement of $I(g)$.
\end{not*}
\begin{defn} Two sets are \emph{almost equal} if their symmetric difference is finite.
\end{defn}
For any $g \in \G$ and any infinite orbit $\Omega$ of $g$, our aim is now to describe a set almost equal to $\Omega$, so to have a suitable description of the infinite orbits of elements of $\G$. We work with $t_{[k]}(g)>0$, since if $t_{[k]}(g)$ is negative, we will be able to apply our arguments to $g^{-1}$. For any $i_1 \in [k]_g$ and $m_1\ge z_{i_1}(g)$, we shall compute the orbit of $(i_1, m_1)$ under $g$. First,
\[\{(i_1, m_1)g^{d|[k]_g|}\mid d \in \N\}=\{(i_1,m)\mid m>m_1, m\equiv m_1 \bmod|t_{[k]}(g)|\}.\] 
Similarly, $\{(i_1, m_1)g^{d|[k]_g|+1}\mid d \in \N\}$ is equal to
\begin{align}\label{eqnimageofg2}
\{(i_1\sigma_g,m)\mid m>m_1+t_{i_1}(g), m\equiv m_1+t_{i_1}(g) \bmod |t_{[k]}(g)|\}.
\end{align}
For every $2\le s\le |[k]_g|$, setting $i_s:=i_1\sigma_g^{s-1}$ and $m_s:=m_1+\sum\limits_{d=1}^{s-1}t_{i_d}(g)$ we have, for any $0\le r<|[k]_g|$, that
\[\{(i_1, m_1)g^{d|[k]|_g+r}\mid d \in \N\}=\{(i_{r+1}, m)\mid m>m_{r+1}, m\equiv m_{r+1} \bmod |t_{[k]}(g)|\}\]
and hence $\{(i_1, m_1)g^d\mid d \in \N\}$ is equal to
\begin{align}\label{eqnintrosetsX}
\bigsqcup\limits_{q=1}^{|[k]_g|}\{(i_q, m)\mid m>m_q, m\equiv m_q \bmod |t_{[k]}(g)|\}.
\end{align}

It is therefore natural to introduce the following.
\begin{defn}\label{defn-setsXid} Let $g \in \G$, $i_1 \in I(g)$, and $m_1 \in \N$. Then
 \[X_{i_1, m_1}(g):=\{(i_1, m)\in X_n\mid m\equiv m_1 \bmod{|t_{[i_1]}(g)|}\}.\]
Furthermore, for every $2\le s<|[k]_g|$ let $i_s:=i_1\sigma_g^{s-1}$, $m_s:=m_1+\sum\limits_{d=1}^{s-1}t_{i_d}(g)$, and
\[X_{[i_1],m_1}(g):=\bigsqcup\limits_{q=1}^{|[i_1]_g|}X_{i_q, m_q}(g).\]
Note that suppressing $m_2, \ldots, m_{|[k]_g|}$ from the notation is not ambiguous since these are uniquely determined from $i_1, m_1$, and $g$.
\end{defn}
Let us summarise what we have shown.
\begin{lem} Let $g \in \G$ and $i \in I(g)$. Then, for any infinite orbit $\Omega$ of $g$ intersecting $\{(i, m)\in X_n\mid m\ge z_i(g)\}$, there exists $i' \in I(g)$ and constants $d_1, e_1 \in \N$ such that the set
 \[X_{[i], d_1}(g)\sqcup X_{[i'], e_1}(g)\]
is almost equal to $\Omega$.
\end{lem}

\subsection{Identities arising from the equation for conjugacy}\label{sectionidentitiesfromconjugacy}
In Section \ref{Sectionequivproblem} we showed that TCP($H_n$) was equivalent to producing an algorithm which, given any $a, b \in \G$, decides whether or not $a$ and $b$ are $H_n$-conjugated. In this section we shall show some necessary conditions which any $x \in H_n$ must satisfy in order to conjugate $a$ to $b$. First, some simple computations to rewrite $t_i(\sigma_ax\sigma_a^{-1})$ are needed. Note that, since $\sigma_a=\sigma_b$ is a necessary condition for $a$ and $b$ to be $H_n$-conjugated (and $\sigma_a, \sigma_b$ are computable), the following will not be ambiguous.
\begin{not*} We will write $[i]$ to denote $[i]_a$ (which is also a class of $b$).
\end{not*}

\begin{lem} For any isometric permutation of the rays $\sigma$ and any $y \in H_n$,
\[t_i(\sigma^{-1}y\sigma)=t_{i\sigma^{-1}}(y)\;\text{for all}\;i \in \Z_n.\]
\end{lem}
\begin{proof}
Let $\sigma=\sigma_1\sigma_2\ldots\sigma_s$ be written in disjoint cycle notation, and let $\sigma_1=(i_1\;i_2\;\ldots\;i_q)$. Since $\sigma \in N_{\Sym(X_n)}(H_n)$, we have that $\sigma^{-1}y\sigma \in H_n$. We compute $t_{i_1}(\sigma^{-1}y\sigma)$ by considering the image of $(i_1, m)$ where $m\ge \max\{z_i(y)\mid i\in \Z_n\}$:
\begin{align*}
(i_1,m)\sigma^{-1}y\sigma=(i_q,m)y\sigma=(i_q,m+t_{i_q}(y))\sigma=(i_1,m+t_{i_q}(y)).
\end{align*}
Similarly, for $1<s\le q$,
\begin{align*}
(i_s,m)\sigma^{-1}y\sigma=(i_{s-1},m)y\sigma=(i_{s-1},m+t_{i_{s-1}}(y))\sigma=(i_s,m+t_{i_{s-1}}(y)).
\end{align*}
Thus $t_i(\sigma^{-1}y\sigma)=t_{i\sigma^{-1}}(y)$ for any $i \in \Z_n$.
\end{proof}

\begin{lem}\label{lem-sametranslations} Let $a, b \in \G$. Then a necessary condition for $a$ and $b$ to be $H_n$-conjugated is that, for all $[i]$ classes, $t_{[i]}(a)=t_{[i]}(b)$. Also if $x\in H_n$ conjugates $a$ to $b$, then $t_i(\omega_b)=-t_i(x)+t_i(\omega_a)+t_{i\sigma_a}(x)$ for all $i\in \Z_n$.
\end{lem}
\begin{proof} From our hypotheses, we have that
\begin{align*}
\omega_b\sigma_a&=x^{-1}\omega_a\sigma_ax\\
\Rightarrow\omega_b&=x^{-1}\omega_a(\sigma_ax\sigma_a^{-1})
\end{align*}
and so, since $\omega_a, \omega_b, x$, and $\sigma_ax\sigma_a^{-1}$ can all be considered as elements of $H_n$, we have for all $i \in \Z_n$ that
\[t_i(\omega_b)=t_i(x^{-1})+t_i(\omega_a)+t_i(\sigma_ax\sigma_a^{-1})\]
which from the previous lemma can be rewritten as
\begin{align*}
t_i(\omega_b)=-t_i(x)+t_i(\omega_a)+t_{i\sigma_a}(x).
\end{align*}
Now, for any branch $i'$, we sum over all $k$ in $[i']$
\begin{align*}
\sum\limits_{k \in [i']}t_k(\omega_b)&=-\sum\limits_{k \in [i']}t_k(x)+\sum\limits_{k \in [i']}t_k(\omega_a)+\sum\limits_{k \in [i']}t_k(x)\\
\Rightarrow\sum\limits_{k \in [i']}t_k(\omega_b)&=\sum\limits_{k \in [i']}t_k(\omega_a)\\
\Rightarrow t_{[i']}(b)&=t_{[i']}(a)
\end{align*}
as required.
\end{proof}
Thus, if $a, b \in \G$ are $H_n$-conjugated, then $I(a)=I(b)$. For any $g \in \G$ the set $I(g)$ is computable, and so the first step of our algorithm can be to check that $I(a)=I(b)$. Hence the following is not ambiguous.
\begin{not*} We shall use $I$ to denote $I(a)$ and $I(b)$.
\end{not*}

\begin{lem}\label{constantsA} Let $a, b \in \G$ be conjugate by $x \in H_n$. Then, for each class $[k]$, choosing a value for $t_{i'}(x)$ for some $i' \in [k]$ determines values for $\{t_i(x)\mid i \in [k]\}$. Moreover, let $i_1 \in [k]$ and, for $2\le s\le |[k]|$, let $i_s:=i_1\sigma_a^{s-1}$. Then the following formula determines $t_{i_s}(x)$ for all $s \in \Z_{|[k]|}$
\[t_{i_s}(x)=t_{i_1}(x)+\sum\limits_{d=1}^{s-1}(t_{i_d}(\omega_b)-t_{i_d}(\omega_a)).\]
\end{lem}
\begin{proof}
If $|[k]|=1$, then there is nothing to prove. Lemma \ref{lem-sametranslations} states that
\begin{align*}%\label{eqn-lemtranslation2}
t_i(\omega_b)=-t_i(x)+t_i(\omega_a)+t_{i\sigma_a}(x)
\end{align*}
for all $i\in \Z_n$, and so we are free to let $i:=i_s \in [k]$, where $s\in \Z_{|[k]|}$, to obtain
\begin{align*}
t_{i_s}(\omega_b)&=-t_{i_s}(x)+t_{i_s}(\omega_a)+t_{i_s\sigma_a}(x)\\
\Rightarrow t_{i_s\sigma_a}(x)&=t_{i_s}(x)+t_{i_s}(\omega_b)-t_{i_s}(\omega_a).
\end{align*}
Setting $s=1$ provides a formula for $t_{i_2}(x)$. If $2\le s< |[k]|$, then
\begin{align*}
t_{i_s\sigma_a}(x)&=t_{i_s}(x)+t_{i_s}(\omega_b)-t_{i_s}(\omega_a)\\
&=t_{i_{s-1}}(x)+t_{i_{s-1}}(\omega_b)-t_{i_{s-1}}(\omega_a)+t_{i_s}(\omega_b)-t_{i_s}(\omega_a)\\
&=\ldots\\
&=t_{i_1}(x)+\sum\limits_{d=1}^s(t_{i_d}(\omega_b)-t_{i_d}(\omega_a)).
\end{align*}
Thus, for all $s \in \Z_{|[k]|}$, we have a formula for $t_{i_s}(x)$ which depends on the computable values $\{t_i(a), t_i(b)\mid i \in [k]\}$ and the value of $t_{i_1}(x)$.
\end{proof}

\section{An algorithm for finding a conjugator in $H_n$}\label{sectionTCPHnalgorithm}
In this section we construct an algorithm which, given $a, b \in \G$ with $\sigma_a=\sigma_b$, either outputs an $x \in H_n$ such that $x^{-1}ax=b$, or halts without outputting such an $x$ if one does not exist.

We will often need to make a choice of some $i \in [k]_g$. For each class $[k]_g$ we shall do this by introducing an ordering on $[k]_g$. We shall choose this ordering to be the one defined by $i_1:=\inf [k]_g$ (under the usual ordering of $\N$) and $i_s:=i_1\sigma_g^{s-1}$ for all $2\le s\le |[k]_g|$. Hence $[k]_g=\{i_1,\ldots, i_{|[k]_g|}\}$.
\subsection{An algorithm for finding a conjugator in $\FSym$}\label{sectionFSymalg}
Many of the arguments of this section draw their ideas from \cite[Section 3]{ConjHou}. By definition, any element which conjugates $a$ to $b$ will send the support of $a$ to the support of $b$. If we wish to find a conjugator in $\FSym$, this means that the symmetric difference of these sets must be finite, whilst $\Supp(a) \cap \Supp(b)$ can be infinite.
\begin{not*} For any $g, h \in \Sym$, let $N(g, h):=\Supp(g)\cap\Supp(h)$, the intersection of the supports of $g$ and $h$.
\end{not*}
\begin{not*}
Let $g \in \G$. Then $Z(g):= \{(i,m)\in X_n\mid i \in \Z_n$ and $m<z_i(g)\}$
which is analogous to the $Z$ region used by some authors when dealing with the Houghton groups.
\end{not*}
This yields a quadratic solution to WP($\G$). First note that, for any $i\in \Z_n$, $(i,z_i(g))g=(i\sigma_g,z_i(g)+t_i(g))$. Thus if $g$ acts trivially on $\{(i,z_i(g))\mid i \in \Z_n\}$, then $\Supp(g)\subseteq Z(g)$. Secondly, by using \cite[Lem 2.1]{ConjHou}, the size of $Z(g)$ is bounded by a linear function in terms of $|g|_S$ (the word length of $g$ with respect to $S$).  Finally, for each point in $Z(g) \cup \{(i,z_i(g))\mid i \in \Z_n\}$, compute the action of $g$ (which requires $|g|_S$ computations). This is sufficient to determine if $g$ is the identity, since we have that $g$ is the identity if and only if $g$ acts trivially on  $Z(g) \cup \{(i,z_i(g))\mid i \in \Z_n\}$.

\begin{defn}\label{defnonlyrcycles} Let $g\in \G$. Then, for a fixed $r \in \N$, let $g_r$ denote the element of $\Sym(X_n)$ which consists of the product of all of the $r$-cycles of $g$. Furthermore, let $g_\infty$ denote the element of $\Sym(X_n)$ which consists of the product of all of the infinite cycles of $g$.
\end{defn}

Our strategy for deciding whether $a, b \in \G$ are $\FSym$-conjugated is as follows. We show, for any $r \in \N$, that if $a_r$ and $b_r$ are $\FSym$-conjugated, then $a_r$ and $b_r$ are conjugate by an $x\in\FSym$ where $\Supp(x)$ is computable. Similarly we show that if $a_\infty$ and $b_\infty$ are $\FSym$-conjugated, then there is a computable finite set such that there is a conjugator of $a_\infty$ and $b_\infty$ with support contained within this set. In order to decide if $a, b \in \G$ are $\FSym$-conjugated we may then decide if $a_\infty$ and $b_\infty$ are $\FSym$-conjugated, produce such a conjugator $y_1$ if one exists, and then decide if $y_1^{-1}ay_1$ and $b$ are $\FSym$-conjugated by deciding whether $(y_1^{-1}ay_1)_r$ and $b_r$ are $\FSym$-conjugated for every $r \in \N$ (which is possible since $b_r$ is non-trivial for only finitely many $r\in \N$, see Lemma \ref{lem-actioncomputable}).

\begin{lem}\label{lem-ConjHou} If $g, h \in \Sym$ are $\FSym$-conjugated, then
\begin{align*} 
  |\Supp{(g)} \setminus  N(g, h)| = |\Supp{(h)} \setminus  N(g, h)|<\infty.
\end{align*}
\end{lem}
\begin{proof} The proof \cite[Lem 3.2]{ConjHou} applies to our more general hypotheses.
\end{proof}

\begin{lem}\label{lemgrinG}Let $g \in \G$ and $r \in \N$. Then $g_r\in \G$. Note this means that $g_r$ restricts to a bijection on $Z(g_r)$ and $X_n\setminus Z(g_r)$.
\end{lem}
\begin{proof} It is clear that $g_r \in \Sym(X_n)$. From our description of orbits in Section \ref{sectionorbits}, for all $(i,m) \in X_n\setminus Z(g)$ we have that $(i,m)$ lies either on: an infinite orbit of $g$; an orbit of $g$ of length $s\ne r$; or on an orbit of $g$ of length $r$. In the first two cases, we have that $(i, m)g_r=(i,m)$ for all $(i,m) \in  X_n\setminus Z(g)$. In the final case, we have that $(i, m)g_r=(i,m)g$ for all $X_n\setminus Z(g)$. Hence $g_r$ is an element for which there exists an isometric permutation of the rays $\sigma$, and constants $t_1(g_r), \ldots, t_n(g_r) \in \Z$ and $z_1(g_r), \ldots, z_n(g_r) \in \N$ such that for all $i \in \Z_n$
\begin{align*}
(i,m)g_r=(i\sigma, m+t_i(g_r))\;\text{for all}\;m\ge z_i(g_r)
\end{align*}
which was labelled (\ref{houextcondsimpler}) in Section \ref{sectionorbits}. Thus, $g_r \in \G$.
\end{proof}

\begin{lem} \label{lem-actioncomputable} Let $S$ be the generating set of $\G$ described in Lemma \ref{lemmagenset}. Then there is an algorithm which, given any $(i, m)\in X_n$ and any word $w$ over $S^{\pm1}$ representing $g \in \G$, computes the set $\mathcal{O}_g(i,m):=\{(i,m)g^d\mid d \in \Z\}$. Moreover the set $\{g_r\ne \id\mid r \in \N\cup\{\infty\}\}$ is finite and computable.

\end{lem}
\begin{proof} From \cite[Lem 2.1]{ConjHou}, the numbers $t_1(\omega_g), \ldots, t_n(\omega_g), z_1(\omega_g),\ldots,z_n(\omega_g)$ are computable. Hence the action of $\sigma_g$ is computable, as are the classes $[i]_g$ and the numbers $z_i(g)$ and $t_{[i]}(g)$.

We note that for all $(i',m') \in X_n$, the number $|\mathcal{O}_g(i',m')|$ is computable. First, let $(i',m') \in X_n\setminus Z(g)$. Then $|\mathcal{O}_g(i',m')|$ is infinite if $t_{[i']}(g)\ne0$ and is equal to $|[i']_g|$ otherwise. If $(i',m') \in Z(g)$, either $|\mathcal{O}_g(i',m')|$ is finite and so $(i',m')g^d=(i',m')$ for some $0\le d \le |Z(g)|$, or $(i',m')g^d \in X_n\setminus Z(g)$ for some $d \in \N$ and so $(i',m')$ lies in an infinite orbit of $g$. When $\mathcal{O}_g(i',m')$ is finite, it is clearly computable.

Now we deal with the infinite orbits of $g$. For every $i \in I(g)$ and $d_1\in \N$, the set $X_{[i], d_1}(g)$ - introduced in Definition \ref{defn-setsXid} - is computable from the numbers $t_1(g),\ldots, t_n(g)$ and classes $\{[k]_g\mid k\in\Z_n\}$. Given $(j,m) \in X_{[i], d_1}(g)$, the set $F_g(j,m):=\{(j,m)g^d\mid d \in \Z\}\cap Z(g)$ is finite and computable. Hence we may compute the $i' \in I(g)$ and $e_1 \in \N$ such that $\mathcal{O}_g(j,m)$ is almost equal to $X_{[i], d_1}(g)\sqcup X_{[i'], e_1}(g)$. Thus
\[\mathcal{O}_g(j,m)=F_g(j,m)\sqcup (X_{[i], d_1}(g)\setminus Z(g))\sqcup (X_{[i'], e_1}(g)\setminus Z(g))\]
and is therefore computable.

Since we can compute the size of orbits and the set $Z(g)$, we may determine those $r \in \N$ for which $g_r$ is non-trivial. First, record those $r$ for which there is an orbit of length $r$ within $Z(g)$. Secondly, for each $k \in I^c(g)$, compute $|[k]|$ (since then $g$ will have infinitely many orbits of length $|[k]|$). Finally compute $|I(g)|$ to determine whether or not $g_{\infty}$ is trivial. Hence the set $\{g_r\ne \id\mid r \in \N\cup\{\infty\}\}$ is finite and computable.
\end{proof}
This lemma means that, given a word $w$ in $S^{\pm1}$ representing $g \in \G$, we may compute the equations (\ref{houextcond}) - which appeared after Definition \ref{classes} - that capture the `eventual' way that $g$ acts. Since $Z(g)$ is finite, only finitely many more equations are required to completely describe the action of $g$ on $X_n$. 

This lemma also means that, given a word $w$ over $S^{\pm1}$ representing $g \in \G$ and any $r \in \N\cup\{\infty\}$, the action of the element $g_r$ can be determined.

\begin{not*} Let $g \in \Sym(X_n)$. Given any set $Y\subseteq X_n$ for which $Yg=Y$ (so that $g$ restricts to a bijection on $Y$), let $g\big|Y$ denote the element of $\Sym(X_n)$ which acts as $g$ on the set $Y$ and leaves all points in $X_n\setminus Y$ invariant i.e.\ for every $(i,m) \in X_n$
\[(i,m)(g\big|Y):=\left\{\begin{array}{ll}(i,m)g&\text{if}\;(i,m) \in Y\\
(i,m)&\text{otherwise.}\end{array}\right.\]
\end{not*}
\begin{not*} Given $g \in \G$, let $Z^*(g):= Z(g)\cap\Supp(g)$.
\end{not*}

\begin{lem} \label{lemraysjinIc} Fix an $r \in \{2, 3, \dots\}$ and let $g, h \in \G$. If $g_r$ and $h_r$ are $\FSym$-conjugated, then there exists an $x \in \FSym$ which conjugates $g_r$ to $h_r$ such that $\Supp(x)\subseteq Z^*(g_r)\cup Z^*(h_r)$.
\end{lem}
\begin{proof} Let $Z:=Z(g_r)\cup Z(h_r)$ and $Z^c:=X_n\setminus Z$. If $g_r$ and $h_r$ are $\FSym$-conjugated, then
\begin{align}\label{eqnactionmatches}
(i, m)g_r=(i,m)h_r\;\text{for all}\;(i,m) \in Z^c.
\end{align}
By Lemma \ref{lemgrinG} and (\ref{eqnactionmatches}), $g_r$ and $h_r$ restrict to a bijection of $Z^c$. Thus, $g_r$ and $h_r$ must also restrict to a bijection of $Z$. We may therefore apply Lemma \ref{lem-ConjHou} to $g_r$ and $h_r$ to obtain
\begin{align*}
|\Supp{(g_r)}\setminus N(g_r, h_r)|=|\Supp{(h_r)}\setminus N(g_r, h_r)|
\end{align*}
which, since $Z^c\subseteq N(g_r, h_r)$, implies that
\begin{align*}
|\Supp{(g_r\big|Z)}\setminus N(g_r\big|Z, h_r\big|Z)|=|\Supp{(h_r\big|Z)}\setminus N(g_r\big|Z, h_r\big|Z)|
\end{align*}
and since $Z$ is finite,
\begin{align*}
|\Supp(g_r\big|Z)|=|\Supp(h_r\big|Z)|<\infty.
\end{align*}
Since $g_r$ and $h_r$ consist of only $r$-cycles, $g_r\big|Z$ and $h_r\big|Z$ are elements of $\FSym$ with the same cycle type, and so are $\FSym$-conjugated. Thus there is a conjugator $x \in \FSym$ with $\Supp(x)\subseteq \Supp(g_r\big|Z)\cup \Supp(h_r\big|Z)=Z^*(g_r)\cup Z^*(h_r)$.
\end{proof}

\begin{lem}\label{rem-Icdone}  If $g_\infty=h_\infty$ and $g, h \in \G$ are $\FSym$-conjugated, then there is an $x \in \FSym$ which conjugates $g$ to $h$ with $\Supp(x)\subseteq Y$, a computable finite set disjoint from $\Supp(g_\infty)$.
\end{lem}
\begin{proof}
By Lemma \ref{lem-actioncomputable}, $\{g_r\ne\id\mid r\in\N\}$ is finite and computable implying that there is a computable $k\in \N$ such that $g_r=\id$ for all $r> k$. Using Lemma \ref{lemraysjinIc}, compute $x_2\in \FSym$ that conjugates $g_2$ to $h_2$ with $\Supp(x_2)\subseteq Z^*(g_2)$. Therefore $(x_2^{-1}gx_2)_\infty=h_\infty$ and $(x_2^{-1}gx_2)_2=h_2$. Let $g^{(2)}:=x_2^{-1}gx_2$. Now, for each $i=3, 4, \ldots, k$, use Lemma \ref{lemraysjinIc} to define $x_i$ to conjugate $(g^{(i-1)})_{i}$ to $h_{i}$ with $\Supp(x_i)\subseteq Z^*((g^{(i-1)})_i)\cup Z^*(h_i)$ and $g^{(i)}:=x_i^{-1}g^{(i-1)}x_i$. Then $\prod_{r=2}^{k}x_r$ conjugates $g$ to $h$ since $(x_2x_3\ldots x_k)^{-1}g(x_2x_3\ldots x_k)=x_k^{-1}(\ldots (x_3^{-1}(x_2^{-1}gx_2)x_3)\ldots)x_k$.
\end{proof}
We now reduce finding an $\FSym$-conjugator of $a$ and $b$ to the case of Lemma \ref{rem-Icdone}. In order to do this we require a well known lemma.
\begin{lem}\label{lemcentraliserlemma} Let $x \in G$ conjugate $a, b \in G$. Then, $x' \in G$ also conjugates $a$ to $b$ if, and only if, $x'=cx$ for some $c \in C_G(a)$.
\end{lem}
\begin{lem} \label{lemraysjinI} Let $g, h \in \G$. If $g_\infty$ and $h_\infty$ are $\FSym$-conjugated, then   there exists an $x\in \FSym$ which conjugates $g_\infty$ to $h_\infty$ with $\Supp(x)\subseteq Z(g_\infty)\cup Z(h_\infty)$.
\end{lem}
\begin{proof} We show that a bound for $z_i(x)$ is computable for each $i\in \Z_n$, inspired in part by the proof of \cite[Prop 3.1]{ConjHou}. Note that $I(g_\infty)=I(g)=I(h)=I(h_\infty)$. For all $i \in I(g_\infty)$ and all $m\ge z_i(g_\infty)$, we have that $(i, m)(g_\infty)^{|[i]_g|}=(i, m+t_{[i]}(g))$.

Our claim is that there is an $x \in \FSym$ which conjugates $g_\infty$ to $h_\infty$ with
\[(i,m)x=(i,m)\;\text{for all}\;i \in \Z_n\;\text{and for all}\;m\ge \max(z_i(g_\infty), z_i(h_\infty)).\]
Let $Z:=Z(g_\infty)\cup Z(h_\infty)$. We first work with $I^c(g_{\infty})$. Let $j \in I^c(g)$, $(j,m) \in X_n\setminus Z$, and assume that $(j,m) \in \Supp(x)$. Then $(j,m)h_\infty=(j,m)$ and so $(j,m)x^{-1}g_\infty x=(j,m)$. If $(j,m)x^{-1} \not\in \Supp(g_\infty)$, then the 2-cycle $\gamma:=((j,m)\;(j,m)x^{-1})$ is in $C_{\FSym}(g_\infty)$ and so $x':=\gamma x$ also conjugates $g_\infty$ to $h_\infty$ by Lemma \ref{lemcentraliserlemma}. Now $(j,m)\gamma x=(j,m)$ and so $(j,m) \not\in \Supp(x')$.  If $(j,m)x^{-1}=(j', m') \in \Supp(g_\infty)$, then $x$ sends $(j',m')$ to $(j,m)$. But, from the fact that $(j,m)x^{-1}g_\infty x=(j,m)$, $x$ must send  $(j',m')g_\infty$ to $(j,m)$, and so $x$ sends both $(j',m')$ and $(j',m')g_\infty$ to $(j,m)$, a contradiction. Hence $\{(j,m)\mid j \in I^c(g)\}\cap \Supp(x)\subseteq Z(g_\infty)\cup Z(h_\infty)$.

We now work with $i \in I(g_\infty)$. Let $t_{[i]}(g)>0$, since replacing $g$ and $h$ with $g^{-1}$ and $h^{-1}$ respectively will provide an argument for $t_{[i]}(g)<0$. Assume, for a contradiction, that for some $i \in I(g_\infty)$ we have that
\[z_i(x)>\max(z_i(g_\infty), z_i(h_\infty)).\]
For all $m\in \{0, 1, \ldots\}$, we have that $z_i(x)+m\ge z_i(x)>\max(z_i(g_\infty), z_i(h_\infty))$. Therefore
\begin{align*}
   &(i,z_i(x)+m)(x^{-1}g_\infty x)^{-|[i]_g|}=(i,z_i(x)+m)(h_\infty)^{-|[i]_g|}\\
   =&(i, z_i(x)+m-t_{[i]}(h_\infty))=(i,z_i(x)+m-t_{[i]}(g_\infty))
\end{align*}
and similarly
\begin{align*}
 &(i,z_i(x)+m)(x^{-1}g_\infty x)^{-|[i]_g|}=(i,z_i(x)+m)x^{-1}(g_\infty)^{-|[i]_g|}x\\
 =&(i,z_i(x)+m)(g_\infty)^{-|[i]_g|}x=(i, z_i(x)+m-t_{[i]}(g_\infty))x
\end{align*}
implying that $(i, z_i(x)+m-t_{[i]}(g_\infty))x=(i,z_i(x)+m-t_{[i]}(g_\infty))$, which contradicts the minimality of $z_i(x)$. 
\end{proof}

\begin{prop}\label{propositionfsym} Given $g, h \in \G$, there is an algorithm which decides whether or not $g$ and $h$ are $\FSym$-conjugated, and produces a conjugator in $\FSym$ if one exists.
\end{prop}
\begin{proof}
By Lemma \ref{lemraysjinI}, enumerating all possible bijections of the set $Z(g_\infty)\cup Z(h_\infty)$ is sufficient to decide whether or not $g_\infty$ and $h_\infty$ are $\FSym$-conjugated. If $g_\infty$ and $h_\infty$ are not $\FSym$-conjugated, then $g$ and $h$ are not $\FSym$-conjugated since such a conjugator would also conjugate $g_\infty$ to $h_\infty$. Let $y_1\in \FSym$ conjugate $g_\infty$ to $h_\infty$ and let $g':=y_1^{-1}gy_1$. By construction, $g'_\infty=h_\infty$.

Lemma \ref{rem-Icdone} states that there is a finite computable set $Y$ such that if $g'$ and $h$ are $\FSym$-conjugated, then there is a conjugator $y_2$ with $\Supp(y_2)\subseteq Y$. Therefore enumerating all possible bijections of $Y$ is sufficient to decide whether $g'$ and $h$ are $\FSym$-conjugated. A suitable conjugator of $g$ and $h$ is then $y_1y_2$. Note that $g$ and $h$ are $\FSym$-conjugated if and only if $g'$ and $h$ are $\FSym$-conjugated.
\end{proof}

\subsection{Reducing the problem to finding a conjugator in $\FSym$}\label{sectionreduce}
The previous section provided us with an algorithm for deciding whether $a, b \in \G$ are $\FSym$-conjugated. Our problem is to decide if $a$ and $b$ are $H_n$-conjugated. We start with Lemma \ref{lemreducetofsym}, a simple observation that allows further application of our algorithm determining conjugacy in $\FSym$.

\begin{defn} Given $g, h \in \G$ and a group $G$ such that $\FSym(X_n)\le G\le H_n$, a \emph{witness set of} $G$-\emph{conjugation} is a finite subset $V(g, h, G)$ of $\Z^n$ satisfying that if $g, h$ are $G$-conjugated, then there is an $x \in G$ such that $g=x^{-1}hx$ and $t(x) \in V(g, h, G)$.
\end{defn}
Note that if $\FSym(X_n)\le U, V\le H_n$ and $\{t(u)\mid u \in U\}=\{t(v)\mid v \in V\}$, then $U=V$. For assume that $u \in U$, $v \in H_n$, and $t(u)=t(v)$. Then $t(u^{-1}v)=\Vec{0}$ which implies that $u^{-1}v=\sigma \in \FSym(X_n)$. Thus $v=u\sigma \in U$, and deciding whether or not a given $h \in H_n$ lies in such a $U$ becomes membership of $t(h)$ in $\{t(u)\mid u \in U\}$.
\begin{lem}\label{lemreducetofsym} Let $\FSym(X_n)\le G\le H_n$. If, for any $g, h \in \G$ it is possible to compute a witness set of $G$-conjugation $V(g, h, G)$, then there is an algorithm which, given any $a, b \in \G$, decides whether they are $G$-conjugated.
\end{lem}
\begin{proof}  We use, from the previous section, that there exists an algorithm for deciding if $g, h \in \G$ are $\FSym$-conjugated. For any $\Vec{v} \in \Z^n$ with $\sum_{i=1}^n\Vec{v}_i=0$ let
\begin{align}\label{xv}
x_{\Vec{v}}:=\prod\limits_{i=2}^ng_i^{-\Vec{v}_i}
\end{align}
where the elements $\{g_i\mid i =2,\ldots, n\}$ are our generators of $H_n$ so that $t(x_{\Vec{v}})=\Vec{v}$. Thus, if $g$ and $h$ are $G$-conjugated, then there exists $\Vec{v} \in V(g, h, G)$ and $x \in \FSym$ such that $x_{\Vec{v}}x$ conjugates $g$ to $h$. Now, to decide if $g$ and $h$ are $G$-conjugated, it is sufficient to check whether any of the pairs $\{(x_{\Vec{v}}^{-1}gx_{\Vec{v}}, h)\mid \Vec{v} \in V(g, h, G)\}$ are $\FSym$-conjugated. This is because
\[x^{-1}(x_{\Vec{v}}^{-1}gx_{\Vec{v}})x=h\Leftrightarrow (x_{\Vec{v}}x)^{-1}g(x_{\Vec{v}}x)=h\]
and so a pair is $\FSym$-conjugated if and only if $g$ and $h$ are $G$-conjugated.
\end{proof}

From this lemma, if, for any $g, h \in \G$, the set $V(g, h, H_n)$ was computable (from only $g$ and $h$), then TCP($H_n$) would be solvable. We shall show that solving our problem can be achieved by producing an algorithm to decide if elements are $G$-conjugated where $G$ is a particular subgroup of $H_n$.
\begin{defn} Let $g \in \G$, $|\sigma_g|$ denote the order of $\sigma_g \in S_n$, and
\[H_n^*(g):=\{x \in H_n\mid t_i(x) \equiv 0 \bmod{\big|t_i(g^{|\sigma_g|})\big|}\;\textrm{for all}\;i \in I(g)\},\]
\[H_n(g):=\{x \in H_n\mid t_i(x)\equiv 0 \bmod{|t_{[i]}(g)|}\;\textrm{for all}\;i \in I(g)\}.\]
\end{defn}
Recall that if $a, b \in \G$ are $H_n$-conjugated, then $\sigma_a=\sigma_b$. Also, from Lemma \ref{lem-sametranslations},   $t_{[i]}(a)=t_{[i]}(b)$ for all $i \in \Z_n$. Hence, if $a$ and $b$ are $H_n$-conjugated, then $H_n^*(a)=H_n^*(b)$ and $H_n(a)=H_n(b)$. Note that $H_n^*(g)\le H_n(g)$ since, for any $i \in I(g)$, $|t_{[i]}(g)|=|t_i(g^{|[i]_g|})|$ and $|t_i(g^{|[i]_g|})|$ divides $|t_i(g^{|\sigma_g|})|$. Given $x\in H_n(g)$ and any infinite orbit $\mathcal{O}_g$ of $g$, we have that $(\mathcal{O}_g)x$ and $\mathcal{O}_g$ are almost equal, and hence elements of $H_n^*(g)$ also have this property.
\begin{lem}\label{lemreducedcasesufficient} Assume there exists an algorithm which, given any $g, h \in \G$, decides whether $g$ and $h$ are $H_n^*(g)$-conjugated. Then there exists an algorithm which, given any $g, h \in \G$, decides whether $g$ and $h$ are $H_n$-conjugated. 
\end{lem}
\begin{proof} Given $g \in \G$, construct the set
\[P_g:=\{(\Vec{v}_1,\ldots, \Vec{v}_n) \in \Z^n: 0\le\Vec{v}_i<\big|t_i(g^{|\sigma_g|})\big|\;\text{for all}\;i \in I(g)\;\text{and}\;\Vec{v}_k=0\;\text{otherwise}\}.\]
Note that, for any $y \in \G$, $P_y$ will be finite. Define $x_{\Vec{v}}$ as in (\ref{xv}) above. Note that $H_n=\bigsqcup_{\Vec{v} \in P_g}x_{\Vec{v}}H_n^*(g)$ and so any element of $H_n$ is expressible as a product of $x_{\Vec{v}}$ for some $\Vec{v} \in P_g$ and an element in $H_n^*(g)$. Thus, deciding whether any of the finite number of pairs $\{(x_{\Vec{v}}^{-1}gx_{\Vec{v}}, h)\mid \Vec{v} \in P_g\}$ are $H_n^*(g)$-conjugated is sufficient to decide whether $g$ and $h$ are $H_n$-conjugated.
\end{proof}

\begin{rem}\label{firstpoint} From Lemma \ref{lemreducetofsym}, if for any $g, h \in \G$ a set $V(g, h, H_n^*(g))$ is computable, then it is possible to decide whether any $g, h \in \G$ are $H_n^*(g)$-conjugated. By Lemma \ref{lemreducedcasesufficient}, it will then be possible to decide whether any $g, h \in \G$ are $H_n$-conjugated. From Section \ref{Sectionequivproblem}, this will mean that TCP($H_n$) is solvable.
\end{rem}
In the following two sections we will show that for any $g, h \in \G$ a witness set of $H_n^*(g)$-conjugation is computable from only $g$ and $h$. In Section \ref{sectionbound1} we show that the following is computable.
\begin{not*} Let $n\in \{2, 3, \ldots\}$ and $g, h \in \G$. Let $M_I(g, h)$ denote a number such that, if $g$ and $h$ are $H_n^*(g)$-conjugated, then there is a conjugator $x \in H_n^*(g)$ with
\begin{align*}\label{eqn-absoluteboundforI}
\sum\limits_{i\in I(g)}\!\!|t_i(x)|<M_I(g, h).
\end{align*}
\end{not*}
In Section \ref{sectionbound2} we show that for any $g, h \in \G$, numbers $\{y_j(g,h)\mid j \in I^c(g)\}$ are computable (using only $g$ and $h$) such that if there exists an $x \in H_n^*(g)$ which conjugates $g$ to $h$, then there is an $x' \in H_n^*(g)$ which conjugates $g$ to $h$ such that $t_i(x')=t_i(x)$ for all $i \in I(g)$ and $t_j(x')=y_j(g,h)$ for all $j \in I^c(g)$.

\begin{rem} \label{keypoint} Note that if the numbers $M_I(g, h)$ and $\{y_j(g,h)\mid j \in I^c(g)\}$ are computable using only $g$ and $h$ then the set $V(g, h, H_n^*(g))$ is computable from only $g$ and $h$. This is because defining $V(g, h, H_n^*(g))$ to consist of all vectors $\Vec{v}$ satisfying:
\begin{enumerate}[i)]
\item $\sum_{i \in I}|\Vec{v}_i|<M_I(g, h)$;
\item $\Vec{v}_i=y_i$ for all $i \in I^c(g)$;
\item $\sum_{i=1}^n \Vec{v}_i=0$.
\end{enumerate}
provides us with a finite set such that if $g$ and $h$ are $H_n^*(g)$-conjugated, then they are conjugate by an $x \in H_n^*(g)$ with $t(x) \in V(g, h, H_n^*(g))$.
\end{rem}

\subsection{Showing that $M_I(g, h)$ is computable}\label{sectionbound1} Let $g, h \in \G$ and $x \in H_n^*(g)$ conjugate $g$ to $h$. In this section we will show that a number $M_I(g, h)$ is computable from only the elements $g$ and $h$.

\begin{not*} Let $x \in H_n^*(a)$ conjugate $a, b \in \G$. Then, for each $i \in I$, let $l_i(x)$ be chosen so that $t_i(x)=l_i(x)|t_{[i]}(a)|$. Note, since $t_i(x)\equiv 0 \bmod{\big|t_{[i]}(a)\big|}$ for all $i \in I$, that $l_i(x) \in \Z$ for each $i \in I$.
\end{not*}
Recall that $\sigma_a=\sigma_b$ and that Lemma \ref{lem-sametranslations} tells us that, if $a$ and $b$ are $H_n$-conjugated, then $t_{[i]}(a)=t_{[i]}(b)$ for every $[i]$ class of $\sigma_a$. Thus, for every $i \in I$ and $d_1 \in \N$, we have that $X_{i_1, d_1}(a)=X_{i_1, d_1}(b)$ (the sets from Definition \ref{defn-setsXid}).  Also, $(X_{i_1, d_1}(a))x$ is almost equal to $X_{i_1, d_1}(a)$ since the set $X_{i_1, d_1}(a)$ consists of all points $(i_1,m) \in X_n$ where $m \equiv d_1 \bmod{|t_{[i_1]}(a)|}$. Moreover $(X_{[i_1], d_1}(a))x$ is almost equal to $X_{[i_1], d_1}(a)$ since $X_{[i_1], d_1}(a)$ is the union of sets $X_{i_s, d_s}(a)$ where for each $i_s \in [i_1]$ we have that $(X_{i_s, d_s}(a))x$ is almost equal to $X_{i_s, d_s}(a)$.  

\begin{rem}\label{howxacts} From Lemma \ref{lem-sametranslations}, we have for any $H_n$-conjugated $g, h \in \G$ and any $i \in I(g)$ that $t_{i\sigma_g}(x)-t_i(x)=t_i(\omega_h)-t_i(\omega_g)$. If we assume that $x \in H_n^*(g)$, we then have that $t_i(\omega_g)\equiv t_i(\omega_h)\bmod{|t_{[i]}(g)|}$. Hence if $g$ and $h$ are $H_n^*(g)$ conjugate, then for any infinite orbit $\mathcal{O}_g$ of $g$, there must be an infinite orbit of $h$ which is almost equal to $\mathcal{O}_g$. It will therefore not be ambiguous to omit $a$ and $b$ and simply write $X_{i_1, d_1}$ and $X_{[i], d_1}$.
\end{rem}

\begin{defn}\label{defnequivrel} Given $g \in \G$, we define an equivalence relation $\sim_g$ on $\{[k]_g \mid k\in I(g)\}$ as the one generated by setting $[i]_g\sim_g[j]_g$ if and only if there is an orbit of $g$ almost equal to $X_{[i], d_1}(g) \sqcup X_{[j], e_1}(g)$ for some $d_1, e_1 \in \N$. Writing $[i]\sim_g[j]$ will not be ambiguous since the relation $\sim_g$ will always be used with respect to the $\sigma_g$-classes of $g$. Note that if $a, b \in \G$ are $H_n$-conjugated, then $\sigma_a=\sigma_b$ and $a$ and $b$ produce the same equivalence relation.
\end{defn}

\begin{prop}\label{propmainargument} Let $a, b \in \G$ be $H_n^*(a)$-conjugated. Then there exists a computable constant $K(a, b)$ (computable from only $a$ and $b$) such that for any $x \in H_n^*(a)$ which conjugates $a$ to $b$ and any given $i, j$ where $[i]\sim_a[j]$, we have that $\big||[i]||l_i(x)|-|[j]||l_j(x)|\big| < K(a, b)$.
\end{prop}
\begin{proof}
We follow in spirit the proof of \cite[Prop 4.3]{ConjHou}. For convenience, we introduce notation to describe a set almost equal to $X_{[i], d_1}$.
\begin{not*} For any set $Y\subseteq X_n$, and any $\Vec{q}:=(\Vec{q}(1), \Vec{q}(2), \ldots, \Vec{q}(n)) \in \N^n$, let $Y\big|_{\Vec{q}}:=Y\setminus\{(i, m)\mid i \in \Z_n$ and $m< \Vec{q}(i)\}$.
\end{not*}
We will assume that $a, b \in \G$ and $x\in H_n^*(a)$ are known. Let $i, j \in I$ satisfy $[i]\sim_a[j]$. Then there exist $d_1, e_1 \in \N$ such that $X_{[i], d_1}\sqcup X_{[j], e_1}$ is almost equal to an infinite orbit of $a$ and hence, by Remark \ref{howxacts}, is almost equal to an infinite orbit of $b$. Denote these infinite orbits by $\mathcal{O}_a$ and $\mathcal{O}_b$ respectively. Note that $\mathcal{O}_ax=\mathcal{O}_b$.

Let $\epsilon_k$ be the smallest integers such that
\begin{enumerate}[i)]
\item for all $k \in [i]\cup [j]$, $\epsilon_k\ge \max(z_k(a), z_k(b))$;
\item for all $k \in [i]$, $\epsilon_k\equiv d_k \bmod{|t_{[i]}(a)|}$;
\item for all $k \in [j]$, $\epsilon_k\equiv e_k \bmod{|t_{[j]}(a)|}$
\end{enumerate}
and define $v\in \Z^n$ by 
$v_k=\left\{\begin{array}{ll}\epsilon_k & \mathrm{if}\ k \in [i]\cup [j]\\1 & \mathrm{otherwise.}\end{array}\right.$

We now have that
\begin{align*}
& X_{[i], d_1}\big|_{v}\subseteq X_{[i], d_1}\cap\mathcal{O}_a;\; X_{[i], d_1}\big|_{v}\subseteq X_{[i], d_1}\cap\mathcal{O}_b;\\
& X_{[j], e_1}\big|_{v}\subseteq X_{[j], e_1}\cap\mathcal{O}_a;\;\textrm{and}\;X_{[j], e_1}\big|_{v}\subseteq X_{[j], e_1}\cap\mathcal{O}_b.
\end{align*}
This allows us to decompose $\mathcal{O}_a$ and $\mathcal{O}_b$:
\begin{align*}
\mathcal{O}_a=X_{[i], d_1}\big|_v \sqcup X_{[j], e_1}\big|_v\sqcup S_{i,j}\\
\mathcal{O}_b=X_{[i], d_1}\big|_v \sqcup X_{[j], e_1}\big|_v\sqcup T_{i,j}
\end{align*}
where $S_{i,j}$ and $T_{i,j}$ are finite sets. Define $\epsilon_k'$ to be the smallest integers such that for all $k \in [i]\cup [j]$
\begin{enumerate}[i)]
\item $\epsilon_k'\ge z_k(x)$;
\item $\epsilon_k'\ge \epsilon_k+|t_k(x)|$;
\item $\epsilon_k'\equiv \epsilon_k \bmod{|t_{[k]}(a)|}$
\end{enumerate}
and define $v' \in \Z^n$ by
$v_k'=\left\{\begin{array}{ll}\epsilon_k' & \mathrm{if}\ k \in [i]\cup [j]\\1 & \mathrm{otherwise.}\end{array}\right.$

These conditions for $\epsilon_k'$ imply that $x$ restricts to a bijection from
\begin{align*}
&X_{[i], d_1}\big|_{v'}\;\text{to}\;X_{[i], d_1}\big|_{v'+t(x)}\\
\textrm{and}\;&X_{[j], e_1}\big|_{v'}\;\text{to}\;X_{[j], e_1}\big|_{v'+t(x)}.
\end{align*}
Hence $x$ restricts to a bijection between the following finite sets
\begin{equation}
\big(X_{[i], d_1}\big|_{v}\setminus X_{[i], d_1}\big|_{v'}\big)\sqcup \big(X_{[j], e_1}\big|_{v}\setminus X_{[j], e_1}\big|_{v'}\big)\sqcup S_{i,j}\label{equationfinite1}
\end{equation}
\begin{align}
\textrm{and}\;\big(X_{[i], d_1}\big|_{v}\setminus X_{[i], d_1}\big|_{v'+t(x)}\big)\sqcup \big(X_{[j], e_1}\big|_{v}\setminus X_{[j], e_1}\big|_{v'+t(x)}\big)\sqcup T_{i,j}\label{equationfinite2}.
\end{align}
By definition, $x$ eventually translates with amplitude $t_k(x)=l_k(x)\cdot|t_{[k]}(a)|$ for each $k \in [i]\sqcup[j]$. Thus
\begin{align*}
&\Big|\Big(X_{[i], d_1}|_{v} \setminus X_{[i], d_1}|_{v'+t(x)}\Big)\Big|=\Big|\big(X_{[i], d_1}|_{v} \setminus X_{[i], d_1}|_{v'}\big)\Big| + \sum\limits_{k \in [i]}l_k(x)\\ 
\textrm{and}\;&\Big|\Big(X_{[j], e_1}|_{v} \setminus X_{[j], e_1}|_{v'+t(x)}\Big)\Big|=\Big|\big(X_{[j], e_1}|_{v} \setminus X_{[j], e_1}|_{v'}\big)\Big|+ \sum\limits_{k' \in [j]}l_{k'}(x).
\end{align*}
Now, since (\ref{equationfinite1}) and (\ref{equationfinite2}) have the same cardinality, we have
\begin{align}\label{eqn-prop4.16}
\sum\limits_{k \in [i]}l_k(x) + \sum\limits_{k' \in [j]}l_{k'}(x) + |T_{i,j}| = |S_{i,j}|
\end{align} 
where $|S_{i,j}|$ and $|T_{i,j}|$ are constants computable from only $a$ and $b$. Using Lemma \ref{constantsA} we may rewrite each element of $\{l_k(x)\mid k \in [i]\}$ as a computable constant plus $l_i(x)$ and each element of $\{l_{k'}(x)\mid k \in [j]\}$ as a computable constant plus $l_j(x)$. Let $A_{i,j}$ denote the sum of all of these constants (which `adjust' the values of the translation lengths of $x$ amongst each ${\sigma_a}$ class). Now (\ref{eqn-prop4.16}) becomes
\begin{align*}
  |[i]|l_i(x) + |[j]|l_j(x)+A_{i, j}+|T_{i,j}|=|S_{i,j}|.
\end{align*}
By the generalised triangle inequality we have
\begin{align*}
&|[i]||l_i(x)|\le|[j]||l_j(x)| + |A_{i,j}| + |S_{i,j}| + |T_{i,j}|\\
\text{and}\;&|[j]||l_j(x)|\le|[i]||l_i(x)| + |A_{i,j}| + |S_{i,j}| + |T_{i,j}|.
\end{align*}
Thus
\begin{align*}
 |[i]||l_i(x)|-|[j]||l_j(x)|&\le |A_{i,j}| + |S_{i,j}| + |T_{i,j}|=:C(i, j)\\
 |[j]||l_j(x)|-|[i]||l_i(x)|&\le |A_{i,j}| + |S_{i,j}| + |T_{i,j}|=C(j, i)\\
\Rightarrow \big||[i]||l_i(x)|-|[j]||l_j(x)|\big|&\le C(i, j).
 \end{align*}
We may then complete this process for all pairs of rays $i', j' \in I$ such that there exist $d_1', e_1' \in \N$ such that $X_{[i'], d_1'} \sqcup X_{[j'], e_1'}$ is almost equal to an infinite orbit of $a$. Let $\hat{C}(a, b)$ denote the maximum of all of the $C(i',j')$.

Now, consider if $k, k' \in I$ satisfy $[k]\sim_a[k']$. This means that there exist $k^{(1)}, k^{(2)}, \ldots, k^{(f)} \in I$ and $d_1^{(0)}, d_1^{(1)}, \ldots, d_1^{(f)}, e_1^{(1)}, \ldots, e_1^{(f)}, e_1^{(f+1)} \in \N$ such that for all $p \in \Z_{f-1}$
\[X_{[k], d_1^{(0)}}\sqcup X_{[k^{(1)}], e_1^{(1)}}, X_{[k^{(p)}], d_1^{(p)}}\sqcup X_{[k^{(p+1)}], e_1^{(p+1)}},\;\text{and}\;X_{[k^{(f)}], d_1^{(f)}}\sqcup X_{[k'], e_1^{(f+1)}}\]
are almost equal to orbits of $a$. We wish to bound $\big||[k]||l_k(x)|-|[k']||l_{k'}(x)|\big|$, and will do this by producing bounds for
\begin{align*}|[k]||l_k(x)|-|[k']||l_{k'}(x)|\;\text{and}\;|[k']||l_{k'}(x)|-|[k]||l_{k}(x)|.\end{align*}
We start by rewriting $|[k]||l_k(x)|-|[k']||l_{k'}(x)|$ as
\begin{align*}
\big(|[k]||l_k(x)|-|[k^{(1)}]||l_{k^{(1)}}(x)|\big)+\big(|[k^{(1)}]||l_{k^{(1)}}(x)|-|[k^{(2)}]||l_{k^{(2)}}(x)|\big)+\ldots\\
\ldots+\big(|[k^{(f-1)}]||l_{k^{(f-1)}}(x)|-|[k^{(f)}]||l_{k^{(f)}}(x)|\big)+\big(|[k^{(f)}]||l_{k^{(f)}}(x)|-|[k']||l_{k'}(x)|\big)
\end{align*}
which by definition is bounded by
\begin{align*}
C(k, k^{(1)})+\sum\limits_{q=1}^{f-1}C(k^{(q)}, k^{(q+1)})+C(k^{(f)}, k')
\end{align*}
and so
\begin{align*}
|[k]||l_k(x)|-|[k']||l_{k'}(x)|&\le C(k, k^{(1)})+\sum\limits_{q=1}^{f-1}C(k^{(q)}, k^{(q+1)})+C(k^{(f)}, k')\\
&\le (f+1)\hat{C}(a, b)\\
&\le n\cdot \hat{C}(a, b)
\end{align*}
Similarly, $|[k']||l_{k'}(x)|-|[k]||l_k(x)|\le (f+1)\hat{C}(a, b)\le n\cdot \hat{C}(a, b)$. Thus $n\cdot \hat{C}(a, b)+1$ is a suitable value for $K(a, b)$.

Now we note that, without knowledge of the conjugator $x$, for all $k,k'$ such that $X_{[k], d_1}\sqcup X_{[k'], e_1}$ is almost equal to an infinite orbit of $a$, the sets $S_{k,k'}$ and $T_{k,k'}$ are computable, and so the constants $C(k, k')$ are also computable. Hence $\hat{C}(a, b)$ and so $K(a, b)$ are computable using only the elements $a$ and $b$.
\end{proof}
We shall now show that, if $a$ is conjugate to $b$ in $H_n^*(a)$, then there is a conjugator $x \in H_n^*(a)$ such that for all $i \in I$ there exists a $j \in I$ such that $[j]\sim_a[i]$ and $t_j(x)=0$. This will allow us to use the previous proposition to bound $|t_k(x)|$ for all $[k]\sim_a[j]$. We will produce such a conjugator using an adaptation of the element defined in \cite[Lem 4.6]{ConjHou}. As with their argument, we again use Lemma \ref{lemcentraliserlemma}, which stated that if $x \in G$ conjugates $a$ to $b$ then $y \in G$ also conjugates $a$ to $b$ if and only if there exists a $c \in C_G(a)$ such that $cx=y$.

\begin{not*} Let $g \in \G$ and $i \in I(g)$. Then $\mathfrak{C}_{g}([i]):=\{k\mid [k]\sim_g[i]\}\subseteq I(g)$. This is the set of all $k \in \Z_n$ corresponding to rays of $X_n$ whose $\sigma_g$-class is related to $[i]_g$.
\end{not*}
This following definition is inspired by \cite[Lem 4.6]{ConjHou}.
\begin{defn}\label{defn-centraliserelt} Let $h \in \G$, $i \in I(h)$, and $g\in C_{H_n}(h)$. Then $g_{[i]}\in \Sym(X_n)$ is defined to be equal to the product of all cycles of $g_\infty$ whose support have infinite intersection with a set $X_{j, d}(g)$ where $j \in \mathfrak{C}_{h}([i])$ and $d \in \N$.
\end{defn}

\begin{lem} \label{5} Let $n\in \{2, 3, \ldots\}$, $h \in  \G$, $g \in  C_{H_n} (h)$ and $i \in  I$. If $t_i(g)\ne 0$, then
$t_j(g)\ne 0$ for all $j \in  \mathfrak{C}_h([i])$.
\end{lem}
\begin{proof}If $j \in  [i]$, then by Lemma \ref{constantsA}, $t_j(g) = t_i(g)\ne 0$.

Since the equivalence relation $\sim_h$ is generated by pairs of classes $[i]\ne [j]$ with $X_{[i], d}(h)\sqcup X_{[j],e}(h)$ for some $d, e \in  \N$ almost equal to some infinite orbit of $h$, it is enough to prove the lemma in that case.
Let $(i,m) \in  X_{[i],d}$ with $m >z_i(g)$ and let $|s| >> 0$ so that $(i, m)h^s = (j', m') \in  X_{[j],e}$ where $m'>z_{j'}(g)$. If $t_{j'}(g) = 0$, then $(i,m)h^sgh^{-s} = (i,m)\ne (i,m + t_i(g)) = (i,m)g$. Hence $t_{j'}(g)\ne 0$ and, therefore, $t_k(g)\ne 0$ for all $k \in  [j']=[j]$.
%***
\end{proof}
\begin{lem}\label{makedefnwork} Let $h\in \G$, $i \in I(h)$, and $g\in C_{H_n}(h)$. Then $g$ has no finite orbits within $\Supp((h^{|\sigma_h|})_{[i]})$. Also, if $t_j(g)\ne0$ for some $j\in \mathfrak{C}_{h}([i])$, then $\Supp(g_{[i]})=\Supp((h^{|\sigma_h|})_{[i]})$.
\end{lem}
\begin{proof} Consider if $\{(k,m)g^d\mid d\in \Z\}$ is finite for some $(k,m)\in \Supp((h^{|\sigma_h|})_{[i]})$. Then, for all $e\in \Z$, $\{(k,m)h^{-e}g^dh^e\mid d\in \Z\}$ is finite since $h^{-e}g^dh^e=g^d$. But this implies that $g$ has infinitely many finite orbits, contradicting that $g\in H_n$.

From now on fix a $j \in\mathfrak{C}_{h}([i])$ and $m>\max\{z_j(g), z_j(h)\}$. By Lemma \ref{5} if $t_j(g)\ne0$ for some $j \in \mathfrak{C}_{h}([i])$, then $t_k(g)\ne0$ for all $k\in \mathfrak{C}_{h}([i])$. Consider if $(j,m)g^d\not\in \Supp(h_{\infty})$ for some $d\in \Z$. Then there exists $e\in \N$ such that $(j,m)g^dh^eg^{-d}=(j,m)$ and because $g\in C_{H_n}(g)$ this is a contradiction: $(j,m)g^dh^eg^{-d}=(j,m)h^e\ne(j,m)$.  

Now consider if $(j,m)g^d\not\in \Supp((h^{|\sigma_h|})_{[i]})$. From the previous two paragraphs, we may assume that $(j,m)g^d$ lies in an infinite orbit of $h$ and an infinite orbit of $g$. Hence we may choose $|d|>>0$ so that $(j,m)g^d=(j',m')$ where $j'\not\in \mathfrak{C}_{h}([i])$ and $m'>\max\{z_{j'}(g), z_{j'}(h)\}$. If $ft_{[j']}(h)>0$ and $ft_{[j']}(h)-dt_{j'}(g)>0$, then
\[(j,m)g^dh^{f|[j']|}g^{-d}=(j',m' +ft_{[j']}(h)-dt_{j'}(g))\]
and so $\{(j,m)g^dh^{f|[j']|}g^{-d}\mid f\in \Z\}$ has infinite intersection with $R_{j'}$. But this contradicts that $j'\not\in\mathfrak{C}_{h}([i])$ i.e.\ that $\{(j,m)h^{|[j']|f}\mid f \in \Z\}\cap R_{j'}$ is finite. Hence $\Supp(g_{[i]})\subseteq\Supp((h^{|\sigma_h|})_{[i]})$. Assume that $\Supp((h^{|\sigma_h|})_{[i]})\not\subseteq \Supp(g_{[i]})$. Then there exists $e\in \Z$ such that $(j,m)h^e\not\in \Supp(g)$. But then $(j,m)h^egh^{-e}=(j,m)$ implies that $(j,m)g=(j,m)$, contradicting that $(j,m)\in\Supp(g)$.
\end{proof}

\begin{lem}\label{lemmacentraliser} Let $h \in \G$ and $g\in C_{H_n}(h)$. Then, for every $i\in I(h)$, $g_{[i]} \in C_{H_n}(h)$. Moreover, $t_j(g_{[i]})=t_j(g)$ if $j \in \mathfrak{C}_h([i])$ and $t_j(g_{[i]})=0$ otherwise.
\end{lem}
\begin{proof}
If $t_j(g)=0$ for some $j\in \mathfrak{C}_h([i])$, Lemma \ref{5} states that $t_k(g)=0$ for all $k \in \mathfrak{C}_h([i])$, and hence $g_{[i]}=\id_{\Sym(X_n)}$ and the statement trivially holds. We are therefore free to apply Lemma \ref{makedefnwork} so that $\Supp(g_{[i]})=\Supp((h^{|\sigma_h|})_{[i]})$. Thus for each $k\not\in \mathfrak{C}_h([i])$, $\Supp(g_{[i]})\cap R_k$ is finite and for each $k \in \mathfrak{C}_h([i])$ the action of $g_{[i]}$ can only differ from the action of $g$ on a finite subset of $R_k$. It follows that $g_{[i]}\in H_n$ and that $t_j(g_{[i]})=t_j(g)$ if $j \in \mathfrak{C}_h([i])$ and $t_j(g_{[i]})=0$ otherwise. Finally we show $h$ and $g_{[i]}$ commute. If $(k,m)\in \Supp((h^{|\sigma_h|})_{[i]})$, then $(k,m)\in \Supp(g)$ and $(k,m)\in \Supp(g_{[i]})$. Hence $(k,m)gh=(k,m)hg\Rightarrow (k,m)g_{[i]}h=(k,m)hg_{[i]}$. If $(k,m)\not\in \Supp((h^{|\sigma_h|})_{[i]})$, then $(k,m), (k,m)h\not\in \Supp(g_{[i]})$.
\end{proof}
From this lemma, if $h \in \G$ and $i\in I(h)$, then $(h^{|\sigma_h|})_{[i]} \in C_{H_n}(h)$.
\begin{lem} \label{lemmaonetixis0}
Let $a, b \in \G$ be conjugate by some $x \in H_n^*(a)$. Then there exists a conjugator $x' \in H_n^*(a)$ which, for each $i \in I$, there is a $j$ such that $[j]\sim_a[i]$ and $t_j(x')=0$.
\end{lem}
\begin{proof} Let $x \in H_n^*(a)$ conjugate $a$ to $b$. From the definition of $H_n^*(a)$, we have for all $i \in I$ that $t_i(x)\equiv 0 \bmod{|t_i(a^{|\sigma_a|})|}$. Thus there exist constants $m_1, \ldots, m_n \in \Z$ such that, for all $i \in I$,
\[t_i(x)=m_i|t_i(a^{|\sigma_a|})|.\]
Let $R(a):=\{j^1,\ldots j^u\}\subseteq I$ be a set of representatives for the $\sim_a$-classes (so that  $a_\infty=\prod_{s=1}^u a_{[j^s]}$). Thus, given $i \in I$, there is a unique $d \in \{1,\ldots,u\}$ such that $[j^d] \sim_a [i]$. Choose some $j \in R(a)$ and consider
\[(a^{-|\sigma_a|m_j})_{[j]}x.\]
By Lemma \ref{lemmacentraliser} $(a^{d|\sigma_a|})_{[j]} \in C_{H_n}(a)$ and $(a^{d|\sigma_a|})_{[j]}\in H_n^*(a)$ for every $d \in \Z$. Now
\[t_j((a^{-|\sigma_a|m_j})_{[j]}x)=0\]
and $(a^{-|\sigma_a|m_j})_{[j]}x$ conjugates $a$ to $b$ by Lemma \ref{lemcentraliserlemma}. Thus a suitable candidate for $x'$ is
\begin{equation*}
\left(\prod\limits_{j \in R(a)}(a^{-|\sigma_a|m_j})_{[j]}\right)x.\qedhere
\end{equation*}
\end{proof}

Recall, given $a, b \in \G$, that $M_I(a, b)$ was a number such that if $a$ and $b$ are $H_n^*(a)$-conjugated, then there exists $x \in H_n^*(a)$ which conjugates $a$ to $b$ with $\sum_{i \in I}|t_i(x)|<M_I(a, b)$.
\begin{prop} Let $a, b \in \G$ be $H_n^*(a)$-conjugated. Then a number $M_I(a, b)$ is computable. \label{PropositiontheboundforiinI}
\end{prop}
\begin{proof} Let $S(a):=\{i^1, \ldots, i^v\}\subseteq I$ be representatives of $I$, so that $\bigsqcup_{i \in S(a)} [i]=I$ and, for any distinct $d, e \in \Z_v$, we have $[i^d]\ne[i^e]$ . We work for a computable bound for $\{|l_i(x)|\mid i \in S(a)\}$ since $t_i(x)=l_i(x)|t_{[i]}(a)|$ and the numbers $|t_{[i]}(a)|$ are computable. Lemma \ref{constantsA} from Section \ref{sectionidentitiesfromconjugacy} will then provide a bound for $|l_i(x)|$ for all $i \in I$. Proposition \ref{propmainargument} says that there is a computable number $K(a, b)=:K$ such that for every $i, j \in I$ where $[i]\sim_a[j]$, we have
\begin{align*}
 \big||[i]||l_i(x)|-|[j]||l_j(x)|\big| < K.
\end{align*}
By Lemma \ref{lemmaonetixis0}, we can assume that for any given $i \in S(a)$, either $t_i(x)=0$ or there exists a $j \in I$ such that $[j]\sim_a[i]$ and $t_j(x)=0$. If $t_i(x)=0$, then we are done. Otherwise,
\begin{align*}
\big||[i]||l_i(x)|-|[j]||l_j(x)|\big| < K\Rightarrow &\big||[i]||l_i(x)|\big| < K \Rightarrow |l_i(x)| < \frac{K}{|[i]|}<K.
\end{align*}
Continuing this process for each $i \in S(a)$ (of which there are at most $n$) implies that
\begin{align*}
 \sum\limits_{i \in S(a)}|l_i(x)|< nK.
\end{align*}
We may then compute, using Lemma \ref{constantsA} from Section \ref{sectionidentitiesfromconjugacy}, a number $K'$ such that
\begin{align*}
 \sum\limits_{i \in I}|l_i(x)|< K'.
\end{align*}
A suitable value for $M_I(a, b)$ is therefore $K'\cdot \max{\{|t_{[i]}(a)|:i \in I\}}$.
\end{proof}

\subsection{Showing that the numbers $\{y_j(g,h)\mid j \in I^c\}$ are computable}\label{sectionbound2} In this section we will show, given any $g, h \in \G$, numbers $\{y_j(g, h)\mid j \in I^c(g)\}$ are computable such that if there exists an $x \in H_n^*(g)$ which conjugates $g$ to $h$, then there is an $x' \in H_n^*(g)$ which conjugates $g$ to $h$ such that $t_i(x')=t_i(x)$ for all $i \in I(g)$ and $t_j(x')=y_j(g, h)$ for all $j \in I^c(g)$.

Note that the condition on elements to be in $H_n^*(g)$ provides no restriction on the translation lengths for the rays in $I^c(g)$. This means that the arguments in this section work as though our conjugator is in $H_n$.

From Section \ref{sectionorbits} we have that for any $g \in \G$, any point $(j,m)$ such that $j \in I^c(g)$ and $m\ge z_j(g)$ lies in an orbit of $g$ of size $|[j]|$.

\begin{not*} Let $g \in \G$ and $r \in \N$. Then $I_r^c(g):=\{j \in I^c(g)\mid |[j]|=r\}$. Also, we may choose $j_r^1,\ldots,j_r^u$ such that $[j_r^1]\cup[j_r^2]\cup\ldots\cup [j_r^u]=I_r^c(g)$ and $[j_r^k]\ne [j_r^{k'}]$ for every distinct $k, k' \in \Z_u$. We shall say that $j_r^1,\ldots,j_r^u$ are representatives of $I_r^c(g)$.
\end{not*}
\begin{lem}\label{introducecjj} Let $a \in \G$. Fix an $r \in \N$, let $j_r^1,\ldots,j_r^u$ be representatives of $I_r^c(a)$, and let $d, d'$ be distinct numbers in $\Z_u$. Fix an ordering on the $r$-cycles of $a$ within $Z(a)$. Label these $\sigma_1, \ldots,\sigma_f$ and, for each $e\in\Z_f$, let $(i_e, m_e)\in \Supp(\sigma_e)$. Now define $c_{j_r^d,j_r^{d'}} \in \Sym(X_n)$ by
\[(i,m)c_{j_r^d,j_r^{d'}}:=\left\{\begin{array}{ll}(i,m+1)&\text{if}\;i \in [j_r^d]\;\text{and}\;m\ge z_i(a)\\
(i,m-1)&\text{if}\;i \in [j_r^{d'}]\;\text{and}\;m\ge z_i(a)+1\\
(i_1, m_1)a^{s-1}&\text{if}\;(i, m)=(j_r^{d'}, z_{j_r^{d'}}(a))a^{s-1}\;\text{for some}\;s\in\Z_r\\
(i_{e+1}, m_{e+1})a^{s-1}&\text{if}\;(i, m)=(i_e, m_e)a^{s-1}\;\text{for some}\;e\in\Z_{f-1},\;s\in\Z_r\\
(j_r^d, z_{j_r^d}(a))a^{s-1}&\text{if}\;(i, m)=(i_f, m_f)a^{s-1}\;\text{for some}\;s\in\Z_r\\
(i,m)&\text{otherwise.}
\end{array}\right.\]
Then $c_{j_r^d,j_r^{d'}} \in C_{H_n}(g)$. 
\end{lem}

\begin{proof} Note that $c_{j_r^d,j_r^{d'}}$ produces a bijection on the $r$-cycles of $g$, and so conjugates $g$ to $g$ i.e.\ $c_{j_r^d,j_r^{d'}} \in C_{\Sym(X_n)}(g)$. By construction it satisfies the condition to be in $H_n$.
\end{proof}
\begin{lem}\label{lem-sametransonIc} Let $g, h \in \G$ and $r \in \N$. If $x_1, x_2 \in H_n$ both conjugate $g$ to $h$ and $j_r^1,\ldots,j_r^u$ are representatives of $I_r^c(g)$, then
\[\sum\limits_{s=1}^ut_{j_r^s}(x_1)=\sum\limits_{s=1}^ut_{j_r^s}(x_2).\]
\end{lem}
\begin{proof} Given $x_1, x_2 \in H_n$ which both conjugate $g$ to $h$, it is possible to produce, by multiplying by an element of the centraliser which is a product of elements $c_{j_r^1,k}$ (with $k \in \{j_r^2,\ldots,j_r^u\}$), elements $x_1', x_2' \in H_n$ which both conjugate $g$ to $h$ and for which $t_{j_r^s}(x_1')=0=t_{j_r^s}(x_2')$ for all $s \in \{2,\ldots, u\}$ and so by Lemma \ref{constantsA}, $t_j(x_1')=t_j(x_2')$ for all $j \in I^c_r(g)\setminus [j_r^1]$. By construction we then have that
\[t_{j_r^1}(x_1')=\sum\limits_{s=1}^ut_{j_r^s}(x_1)\;\text{and}\;t_{j_r^1}(x_2')=\sum\limits_{s=1}^ut_{j_r^s}(x_2).\]
Now consider $y:=x_1'(x_2')^{-1}$. By construction $t_j(y)=0$ for all $j \in I^c_r(g)\setminus[j_r^1]$ since $t_j(x_1')=t_j(x_2')$ for all $j \in I^c_r(g)\setminus[j_r^1]$. Also $t_j(y)=t_{j_r^1}(y)$ for all $j \in [j_r^1]$ since $y$ conjugates $g$ to $g$ (and so we also have that $y \in C_{H_n}(g)$). If $t_{j_r^1}(y)\ne 0$, then $y$ contains an infinite cycle with support intersecting the branch $j_r^1$ and a branch $j_{p} \in I^c(g)\setminus I_r^c(g)$ so that $|[j_{p}]|=p\ne r$ i.e.\ the infinite cycle contains $(j_r^1, m_1), (j_{p}, m_2) \in X_n$. This means there exists an $e \in \Z$ such that $(j_r^1, m_1)y^e=(j_{p}, m_2)$. But then $y^e \in C_{H_n}(g)$ and $y^e$ sends an $r$-cycle to an $p$-cycle where $r\ne p$, a contradiction. Hence for any $x, x' \in H_n$ which both conjugate $g$ to $h$,
\begin{equation*}\sum\limits_{s=1}^ut_{j_r^s}(x)=\sum\limits_{s=1}^ut_{j_r^s}(x').\end{equation*}
\end{proof}
From this proof, the following is well defined.

\begin{not*} Let $g, h \in \G$, $r \in \N$ and $j_r^1,\ldots,j_r^u$ be representatives of $I_r^c(g)$. Then $M_{\{j_r^1,\ldots,j_r^u\}}(g, h)$ denotes the number such that, for any $x \in H_n^*(g)$ which conjugates $g$ to $h$, $\sum_{d=1}^ut_{j_r^d}(x)=M_{\{j_r^1,\ldots,j_r^u\}}(g, h)$. Since we will fix a set of representatives, we will often denote $M_{\{j_r^1,\ldots,j_r^u\}}(g, h)$ by $M_r(g, h)$.
\end{not*}
We will show that one combination $\{y_d\in \Z\mid d \in \Z_u\}$ is computable in order to show, for any $g, h \in \G$ and any $r \in \N$, that $M_r(g, h)$ is computable. The following will be useful for this. Recall that for any $g \in \G$, $Z(g):= \{(i,m)\in X_n\mid i \in \Z_n\;\text{and}\;m<z_i(g)\}$.
\begin{not*} Let $g \in \G$ and $r\in \{2, 3, \ldots\}$. Then $\eta_r(g):=\sfrac{\big|\Supp(g_r\big|Z(g_r))\big|}{r}$ denotes the number of orbits of $g_r\big|Z(g_r)$ of size $r$. This is well defined since Lemma \ref{lemgrinG} states that $g_r$ restricts to a bijection on $Z(g_r)$. Also, let $\eta_1(g):=|Z(g)\setminus (\Supp(g\big|Z(g))|$. Since $Z(g_r)$ is finite for all $r \in \N$, we have that $\eta_r(g)$ is finite for all $r \in \N$.
\end{not*}

\begin{lem}\label{computingyj} Let $g, h \in \G$, $r \in \N$ and $j_r^1,\ldots,j_r^u$ be representatives of $I_r^c(g)$. Then $M_{\{j_r^1,\ldots,j_r^u\}}(g, h)$ and the numbers $\{y_j(g,h) \mid j \in I^c(g)\}$ are computable (using only the elements $g$ and $h$).
\end{lem}
\begin{proof}
For each $r \in \N$, any conjugator of $g$ and $h$ must send the $r$-cycles of $g$ to the $r$-cycles of $h$. Fix an $r\in\N$ and let $j_r^1,\ldots,j_r^u$ be representatives of $I_r^c(g)$. Given any $g, h \in \G$ which are $H_n$-conjugated, let
\begin{align}\label{eqn-secondvalues}
y_k(g, h):=z_k(h)-z_k(g)\;\text{for all}\;k\in \{j_r^2,\ldots,j_r^u\}.\end{align}
\begin{align}\label{eqn-thirdvalues}
y_{j_r^1}(g,h):=z_{j_r^1}(h)-z_{j_r^1}(g)+\eta_r(g)-\eta_r(h).
\end{align}
We work towards proving that the values for $y_k(g,h)$ defined in (\ref{eqn-secondvalues}) and (\ref{eqn-thirdvalues}) are suitable in 3 steps. First, consider if $\eta_r(g)=\eta_r(h)$. This means that there is a conjugator in $\FSym$ which conjugates $g_r\big|Z(g_r)$ to $h_r\big|Z(h_r)$ and hence the values are sufficient. Secondly, consider if $\eta_r(g)=\eta_r(h)+d$ for some $d \in \N$. In this case, first send $\eta_r(h)$ $r$-cycles in $Z(g_r)$ to those in $Z(h_r)$. Then send the $d$ remaining cycles in $Z(g_r)$ to the first $d$ $r$-cycles on the branches $[j_r^1]$ by increasing $y_{j_r^1}(g,h)$ by $d$. Finally, if $\eta_r(g)=\eta_r(h)-e$ for some $e \in \N$, then send the $\eta_r(g)$ $r$-cycles in $Z(g_r)$ to $r$-cycles in $Z(h_r)$ and then send the first $e$ $r$-cycles of $g$ on the branches $[j_r^1]$ to the remaining $r$-cycles in $Z(h_r)$ by decreasing $y_{j_r^1}(g,h)$ by $e$. With all of these cases, the values defined in (\ref{eqn-secondvalues}) and (\ref{eqn-thirdvalues}) are suitable.
\end{proof}

\begin{proof}[Proof of Theorem \ref{main2}] From Remark \ref{firstpoint} and Remark \ref{keypoint} of Section \ref{sectionreduce}, TCP($H_n$) is solvable if, given $a, b \in \G$, the numbers $M_I(a, b)$ and $\{y_j(a, b) \mid j \in I^c\}$ are computable from only $a$ and $b$. The computability of these numbers was shown, respectively, in  Proposition \ref{PropositiontheboundforiinI} and Lemma \ref{computingyj}.
\end{proof}

\section{Applications of Theorem \ref{main2}}
Our strategy is to use \cite[Thm. 3.1]{orbitdecide}. We first set up the necessary notation.
\begin{defn} Let $H$ be a group and $G\unlhd H$. Then $A_{G\unlhd H}$ denotes the subgroup of $\Aut(G)$ consisting of those automorphisms induced by conjugation by elements of $H$ i.e.\ $A_{G\unlhd H}:=\{\phi_h\mid h \in H\}$.
\end{defn}
\begin{defn} Let $G$ be a finitely presented group. Then $A \le \Aut(G)$ is \emph{orbit decidable} if, given any $a,b \in G$, there is an algorithm which decides whether there is a $\phi \in A$ such that $a\phi= b$. If $\Inn(G)\le A$, then this is equivalent to finding a $\phi \in A$ and $x \in G$ such that $x$ conjugates $a\phi$ to $b$.
\end{defn}
The algorithmic condition in the following theorem means that certain computations for $D, E,$ and $F$ are possible. This is satisfied by our groups being given by recursive presentations, and the maps between them being defined by the images of the generators.
\begin{thm}\emph{(}Bogopolski, Martino, Ventura~\cite[Thm. 3.1]{orbitdecide}\emph{)}.\label{thmorbitdecidability}
Let
 $$
1\longrightarrow D \longrightarrow  E \longrightarrow F \longrightarrow 1
 $$
be an algorithmic short exact sequence of groups such that
\begin{itemize}
\item[(i)] $D$ has solvable twisted conjugacy problem,
\item[(ii)] $F$ has solvable conjugacy problem, and
\item[(iii)] for every $1\neq f\in F$, the subgroup $\langle f\rangle$ has finite index in its centralizer $C_F(f)$, and
there is an algorithm which computes a finite set of coset representatives, $z_{f,1},\ldots ,z_{f,t_f}\in F$,
 $$
C_F(f)=\langle f \rangle z_{f,1}\sqcup \cdots \sqcup \langle f \rangle z_{f,t_f}.
 $$
\end{itemize}
Then, the conjugacy problem for $E$ is solvable if and only if the action subgroup $A_{D\unlhd E}=\{\phi_g\mid g \in E\} \le Aut(D)$ is orbit decidable.
\end{thm}

\begin{remR} For all that follows, the action subgroup $A_{D\unlhd E}$ is provided as a recursive presentation where the generators are words from $\Aut(D)$.
\end{remR}
\subsection{Conjugacy for finite extensions of $H_n$}
We shall say that $B$ is a finite extension of $A$ if $A\unlhd B$ and $A$ is finite index in $B$. The following is well known.
\begin{lem}\label{lem-finitegen} If $G$ is finitely generated and $H$ is a finite extension of $G$, then $H$ is finitely generated.
\end{lem}
\begin{prop}\label{extensionsresult} Let $n\in \{2, 3, \ldots\}$. If $E$ is a finite extension of $H_n$, then CP($E$) is solvable.
\end{prop}
\begin{proof} Within the notation of Theorem \ref{thmorbitdecidability}, set $D:=H_n$ and $F$ to be a finite group so to realise $E$ as a finite extension of $H_n$. Conditions (ii) and (iii) of Theorem \ref{thmorbitdecidability} are satisfied since $F$ is finite. Theorem \ref{main2}, the main theorem of the previous section, states that condition (i) is satisfied. Thus CP($E$) is solvable if and only if $A_{H_n\unlhd E}=\{\phi_e\mid e \in E\}$ is orbit decidable. We note that $A_{H_n\unlhd E}$ contains a copy of $H_n$ (since $H_n$ is centreless). Moreover, it can be considered as a group lying between $H_n$ and $\G$. Hence $A_{H_n\unlhd E}$ is isomorphic to a finite extension of $H_n$, and so by Lemma \ref{lem-finitegen} is finitely generated. Thus $A_{H_n\unlhd E}=\langle \phi_{e_1}, \phi_{e_2}, \ldots, \phi_{e_k}\rangle$ where $\{e_1, \ldots, e_k\}$ is a finite generating set of $E$. From Lemma \ref{lem-actioncomputable}, given any $g \in N_{\Sym(X_n)}(H_n)\cong \Aut(H_n)$, we may compute $\sigma_g$: the isometric permutation of the rays induced by $g$. Thus we may compute $\langle \sigma_{e_i}\mid i \in \Z_k \rangle=:E_\sigma$. Now, given $a, b \in H_n$, our aim is to decide whether there exists $\phi_e \in A_{H_n\unlhd E}$ such that $(a)\phi_e=b$. Since $\Inn(H_n)\le A_{H_n\unlhd E}$, this is equivalent to finding a $\tau \in E_\sigma$ and $x \in H_n$ such that $(x\tau)^{-1}a(x\tau)=b$, which holds if and only if $x^{-1}ax=\tau b\tau^{-1}$.

Finally, since $E_\sigma$ is finite (there are at most $n!$ permutations of the rays), searching for an $x \in H_n$ which conjugates $a$ to $\sigma_eb\sigma_e^{-1}$ for all $\sigma_e \in E_\sigma$ provides us with a suitable algorithm. Searching for such a conjugator can be achieved by Theorem \ref{main2} or \cite[Thm. 1.2]{ConjHou}.
\end{proof}
\subsection{Describing $\Aut(U)$ for $U$ finite index in $H_n$}\label{secficonj}
Recall that $g_2,\ldots, g_n$ were elements of $H_n$ such that $g_i$: translates the first branch of $X_n$ by 1; translates the $i$\ts{th} branch by -1; sends $(i,1)$ to $(1,1)$; and does not move any points of the other branches (which meant that, if $n\in\{3, 4, \ldots\}$, then $H_n=\langle g_i\mid i=2,\ldots, n\rangle$). For any given $n\in \{2, 3, \ldots\}$, the family of finite index subgroups $U_p\le H_n$ were defined (for $p\in \N$) in \cite{Hou2} as follows. Note that  $\FAlt(X)$ denotes the index 2 subgroup of $\FSym(X)$ consisting of all even permutations on $X$.
\begin{align*}
U_p := \langle\;\FAlt(X_n), g_i^p\mid i \in \{2,...,n\}\;\rangle
\end{align*}
\begin{not*}
Let $A\le_fB$ denote that $A$ has finite index in $B$.
\end{not*}
Let $n\in \{3, 4, \ldots\}$. If $p$ is odd, then $U_p$ consists of all elements of $H_n$ whose eventual translation lengths are all multiples of $p$. If $p$ is even, then $U_p$ consists of all elements $u$ of $H_n$ whose eventual translations are all multiples of $p$ and
\begin{align}\label{containingFAlt}
u\prod\limits_{i=2}^n g_i^{t_i(u)} \in \FAlt(X_n)
\end{align}
i.e.\ $\FSym(X_n)\le U_p$ if and only if $p$ is odd. This can be seen by considering, for some $i, j \in \Z_n$, the commutator of $g_i^p$ and $g_j^p$. This will produce $p$ $2$-cycles which will produce an odd permutation if and only if $p$ is odd. If $n=2$, then for all $p \in \N$ all $u \in U_p\le H_2$ will satisfy (\ref{containingFAlt}).
\begin{lem}[Burillo, Cleary, Martino, R{\"o}ver~\cite{Hou2}]\label{lemfactsaboutU} Let $n\in \{2, 3, \ldots\}$.
For every finite index subgroup $U$ of $H_n$, there exists a $p\in 2\N$ with
\begin{align*}
 \FAlt(X_n) = U_p'< U_p \le_f U \le_f H_n
\end{align*}
where $U_p'$ denotes the commutator subgroup of $U_p$.
\end{lem}
\begin{proof}[Alternative proof] Let $n\in \{2, 3, \ldots\}$ and let $U\le_fH_n$. Thus $\FAlt\cap U\le_f\FAlt$. Since $\FAlt$ is both infinite and simple, $\FAlt\le U$. Let $\pi_n: H_n\rightarrow \Z^{n-1}$, $g\mapsto (t_2(g),\ldots, t_n(g))$. Thus $(U)\pi_n\le_f\Z^{n-1}$ and so there is a number $d \in \N$ such that $(d\Z)^{n-1}\le (U)\pi_n$ ($[(U)\pi_n:\Z^{n-1}]$ is one such value for $d$). This means that for any $k \in \Z_n\setminus\{1\}$ there exists a $u\in U$ such that $t_k(u)=-d$, $t_1(u)=d$, and $t_i(u)=0$ otherwise. Moreover, for each $k \in \Z_n\setminus\{1\}$ there is a $\sigma \in \FSym$ such that $g_k^d\sigma \in U$. First, let $n\in \{3, 4, \ldots\}$. Since $\FAlt\le U$, we may assume that either $\sigma$ is trivial or is a $2$-cycle with disjoint support from $\Supp(g_k)$. Thus $(g_k^d\sigma)^2=g_k^{2d} \in U$. If $n=2$, we may assume that $\sigma$ is either trivial or equal to $((1, s)\;(1, s+1))$ for any $s \in \N$. Now, by direct computation, $g_2^d((1, 1)\;(1, 2))g_2^d((1, d+1)\;(1, d+2))=g_2^{2d}$. Thus, for any $n\in \{2, 3, \ldots\}$,
\[\langle g_2^{2d},\ldots, g_n^{2d}, \FAlt(X_n)\rangle\le U.\]
Hence, if $p:=2d$, then $U_p\le U$.
\end{proof}
\begin{remR} $\{(U_p)\pi_n\mid p \in \N\}$ are the congruence subgroups of $\Z^{n-1}$.
\end{remR}
Now, given $U\le_fH_n$, our strategy for showing that CP($U$) is solvable is as follows. First, we show for all $p \in \N$ that TCP($U_p)$ is solvable. Using Theorem \ref{thmorbitdecidability}, we then obtain that all finite extensions of $U_p$ have solvable conjugacy problem. By the previous lemma, we have that any finite index subgroup $U$ of $H_n$ is a finite extension of some $U_p$ (note that $U_p\unlhd U$ since $U_p\unlhd H_n$). This will show that CP($U$) is solvable.

TCP($U_p$) requires knowledge of $\Aut(U_p)$. From \cite[Prop. 1]{cox2}, we have that any group $G$ for which there exists an infinite set $X$ where $\FAlt(X)\le G\le \Sym(X)$ has $N_{\Sym(X)}(G)\cong \Aut(G)$ by the map $\rho\mapsto \phi_\rho$. By Lemma \ref{lemfactsaboutU} any finite index subgroup of $H_n$ contains $\FAlt(X_n)$. Thus, if $U\le_f H_n$, then $N_{\Sym(X_n)}(U)\cong \Aut(U)$ by the map $\rho\mapsto \phi_\rho$. In fact we may show that a stronger condition holds.

\begin{lem}\label{lem-monolithic} If $1\ne N\unlhd H_n$, then $\FAlt(X_n)\le N$.
\end{lem}
\begin{proof} We have that $N \cap \FAlt(X_n) \unlhd H_n$. Since $\FAlt(X_n)$ is simple, the only way for our claim to be false is if $N \cap \FAlt(X_n)$ were trivial. Now, $N\le H_n$, and so $[N,N]\le \FSym(X_n)$. Thus $[N,N]$ must be trivial, and so $N$ must be abelian. But the condition for elements $\alpha, \beta \in \Sym(X_n)$ to commute (that, when written in disjoint cycle notation, either a power of a cycle in $\alpha$ is a power of a cycle in $\beta$ or the cycle in $\alpha$ has support outside of $\Supp(\beta)$) is not preserved under conjugation by $\FAlt(X_n)$, and so cannot be preserved under conjugation by $H_n$ i.e.\ $N$ is not normal in $H_n$, a contradiction.
\end{proof}

\begin{remR} It follows that all Houghton groups are monolithic: each has a unique minimal normal subgroup which is contained in every non-trivial normal subgroup. The unique minimal normal subgroup in each case will be $\FAlt(X_n)$.
\end{remR}
We now introduce notation to help describe any finite index subgroup of $H_n$ (where $n\in \{2, 3, \ldots\}$).
\begin{not*} For each $i \in \Z_n$, let $T_i(U):=\min\{t_i(u)\mid u \in U$ and $t_i(u)>0\}$. Furthermore for all $k \in \Z_n$, let $T^k(U):=\sum_{i=1}^kT_i(U)$ and let $T^0(U):=0$.
\end{not*}
We will now introduce a bijection $\phi_U: X_n\rightarrow X_{T^n(U)}$. This bijection will induce an isomorphism $\hat{\phi}_U: \Sym(X_n)\rightarrow \Sym(X_{T^n(U)})$ which restricts to an isomorphism 
\[N_{\Sym(X_n)}(U)\rightarrow N_{\Sym(X_{T^n(U)})}((U)\hat{\phi}_U).\]
Our bijection $\phi_U$  will send the $i$\ts{th} branch of $X_n$ to $T_i(U)$ branches in $X_{T^n(U)}$. For simplicity let $g_1:=g_2^{-1}$. Now, for any $i \in \Z_n$ and $d \in \N$,
\[X_{i, d}(g_i^{T_i(U)})=\{(i,m)\mid m\equiv d \bmod{|t_i(g_i^{T_i(U)})|}\}\]
where $|t_i(g_i^{T_i(U)})|=T_i(U)$ by the definition of $g_i$.

Thus the $i$\ts{th} branch of $X_n$ may be partitioned into $T_i(U)$ parts:
\[X_{i, 1}(g_i^{T_i(U)})\sqcup X_{i, 2}(g_i^{T_i(U)})\sqcup \ldots\sqcup X_{i, T_i(U)}(g_i^{T_i(U)}).\]
We will now define the bijection $\phi_U$ by describing the image under $\phi_U$ of all points in each set $X_{i, d}(g_i^{T_i(U)})$ where $i \in \Z_n$ and $d \in \Z_{T_i(U)}$. Let $(i,m) \in X_{i,d}(g_i^{T_i(U)})$. Then \[((i,m))\phi_U:=\left(T^{i-1}(U)+d,\;\frac{m-d}{T_i(U)}+1\right)\]
i.e.\ $\phi_U$ sends, for all $i \in \Z_n$ and $d \in \Z_{T_i(U)}$, the ordered points of $X_{i, d}(g_i^{T_i(U)})$ to the ordered points of the $(T^{i-1}(U)+d)$\ts{th} branch of $X_{T^n(U)}$. An example of this bijection with $n=T_1(U)=T_2(U)=T_3(U)=3$ is given below.

\begin{figure}[h]
        \captionsetup[subfigure]{aboveskip=4pt,belowskip=-4pt}
        \begin{subfigure}[h]{0.3\textwidth}
\begin{tikzpicture}[scale=0.4]
\foreach \x in {90,210,330}
{
\draw (\x:5cm)  -- (\x:0.5cm);
\filldraw (\x:4.5cm) circle (2pt)-- (\x:0.5cm) circle (2pt);
\filldraw (\x:4cm) circle (2pt)-- (\x:1cm) circle (2pt);
\filldraw (\x:3.5cm) circle (2pt)-- (\x:1.5cm) circle (2pt);
\filldraw (\x:3cm) circle (2pt)-- (\x:2.5cm) circle (2pt);
\filldraw (\x:2cm) circle (2pt)-- (\x:0.5cm) circle (2pt);
\draw[draw=red, thick, rotate=-30] (0.2,-0.3) rectangle (1.8,0.3);
\draw[draw=orange, thick, rotate=-30] (1.8,0.3) rectangle (3.2,-0.3);
\draw[draw=yellow, thick, rotate=-30] (3.25,-0.3) rectangle (4.7,0.3);
}
\end{tikzpicture}
   \caption*{The set $X_n$}        \end{subfigure}\hspace{0.25cm}
   \begin{subfigure}[h]{0.3\textwidth}
\begin{tikzpicture}[scale=0.4]
\foreach \x in {80,90,100,200,210,220,320,330,340}
{
\draw (\x:5cm)  -- (\x:0.5cm);
\filldraw (\x:3.7cm) circle (2pt)-- (\x:1.2cm) circle (2pt);
\filldraw (\x:3.7cm) circle (2pt)-- (\x:2.5cm) circle (2pt);
\draw[draw=red, thick, rotate=60] (-0.45,-1.35) rectangle (0.5,-1);
\draw[draw=orange, thick, rotate=60] (-0.7,-2.3) rectangle (0.7,-2.7);
\draw[draw=yellow, thick, rotate=60] (-0.9,-3.45) rectangle (0.9,-3.9);
}
\end{tikzpicture}
\caption*{The set $X_{T^n(U)}$}\end{subfigure}
\caption{Our bijection between $X_n$ and $X_{T^n(U)}$, which can be\\ visualised by rotating the rectangles 90 degrees clockwise.}\label{figurethemap}
\end{figure}
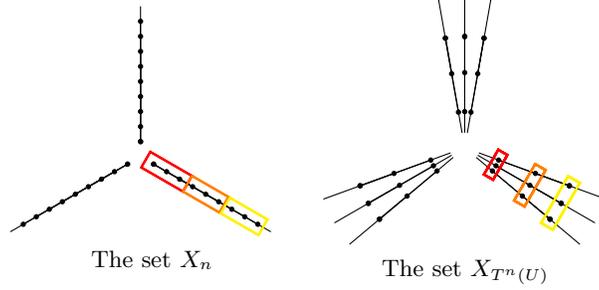

We now describe the image of $U$ under $\hat{\phi}_U$. First, $\hat{\phi}_U$ preserves cycle type. Thus $\FAlt(X_{T^n(U)})\le (U)\hat{\phi}_U$. Moreover $\FSym(X_{T^n(U)})\le (U)\hat{\phi}_U$ if and only if $\FSym(X_n)\le U$. The following will be useful to describe $(U)\hat{\phi}_U$.
\begin{not*}
For any $n \in \N$ and any $i \in \Z_n$, let $R_i:=i \times \N$, the $i$\ts{th} branch of $X_n$ or $X_{T^n(U)}$, and $Q_i:=(R_i)\phi_U$, so that $Q_i$ consists of $T_i(U)$ branches of $X_{T^n(U)}$.
\end{not*}
\begin{lem}\label{factsaboutphihat1} Let $n\in \{2, 3, \ldots\}$ and $U\le_fH_n$. If $u \in U$, then $(u)\hat{\phi}_U\in H_{T^n(U)}$. Moreover, for each $i \in \Z_{T^n(U)}$ there exists a $g \in (U)\hat{\phi}_U$ such that $t_i(g)=1$.
\end{lem}
\begin{proof} Let $n\in \{2, 3, \ldots\}$ and $U\le_fH_n$. Since $\hat{\phi}_U$ preserves cycle type, outside of a finite set, $(u)\hat{\phi}_U$ consists of fixed points and infinite cycles. If $i\in \Z_n$ and $m>z_i(u)+|t_i(u)|$, let  $(i', m')=(i,m)\phi_U$ and note, since $T_i(U)\big|t_i(u)$, that $(i',m')(u)\hat{\phi}_U=(i',m'+t_i(u)/T_i(U))$. Hence $(u)\hat{\phi}_U\in H_{T^n(U)}$. Now, from the definition of $T_i(U)$, there exists a $g \in (U)\hat{\phi}_U$ such that $t_i(g)=1$.
\end{proof}
Using the above notation we have, for any $g \in (U)\hat{\phi}_U$, that
\begin{align}\label{structureofU}
\text{if}\;t_i(g)=d,\;\text{then}\;t_j(g)=d\;\text{for all}\;j\;\text{such that}\;R_j\subseteq Q_i
\end{align}
i.e.\ for any $u \in U$ and any $i \in \Z_n$, the eventual translation lengths of $(u)\hat{\phi}$ for the branches in $Q_i$ must be the same.
\begin{lem}\label{factsaboutphihat} Let $n\in \{2, 3, \ldots\}$ and $U\le_fH_n$. Then $N_{\Sym(X_{T^n(U)})}((U)\hat{\phi}_U)\le H_{T^n(U)}\rtimes R$ where $R\le S_{T^n(U)}$ consists of isometric permutations of the rays of $X_{T^n(U)}$. Thus if $\rho \in N_{\Sym(X_n)}(U)$, then $(\rho)\hat{\phi}_U\in H_{T^n(U)}\rtimes S_{T^n(U)}$. 
\end{lem}
\begin{proof}
Let $G:=N_{\Sym(X_{T^n(U)})}((U)\hat{\phi}_U)$. Describing $G$ will be simpler than describing $N_{\Sym(X_n)}(U)$. We will do this in three stages: we first show, for any $\rho \in G$, that the image under $\rho$ of each ray of $X_{T^n(U)}$ is almost equal to a ray of $X_{T^n(U)}$; secondly we show that $\rho\sigma_\rho^{-1} \in H_{T^n(U)}$, meaning that $G\le H_{T^n(U)}\rtimes S_{T^n(U)}$; and finally we will describe a subgroup of $R\le S_{T^n(U)}$ such that $\{\sigma_\rho\mid\rho \in G\}\le R$.

Let $i,j\in \Z_{T^n(U)}$ and $\rho \in G$. Consider if $(R_i)\rho$ and $R_j$ have infinite intersection but are not almost equal. Either $(R_i)\rho$ has infinite intersection with $R_j$ and $R_{j'}$ where $j'\ne j$, or there is an $i'\ne i$ such that $(R_{i'})\rho$ is almost equal to an infinite subset of $R_j$. For the first case, let $g \in (U)\hat{\phi}_U$ be chosen so that $t_i(g)=1$. Thus $g$ has an infinite cycle containing $\{(i,m)\mid m\ge z_i(g)\}$. Let $(i',m')=(i,z_i(g))\rho$. Then $\{(i',m')(\rho^{-1}g\rho)^d\mid d\in \N\}=\{(i',m')\rho^{-1}g^d\rho\mid d\in \N\}=\{(i,m)\rho\mid m>z_i(g)\}$ which has infinite intersection with $R_j$ and $R_{j'}$. But from the description of orbits of $H_n$ in \cite{ConjHou}, $\rho^{-1}g\rho \not\in H_{T^n(U)}$ and so $(\rho^{-1}g\rho)\hat{\phi}_U^{-1} \not\in U$ i.e.\ $\rho \not\in G$. The second case reduces to the first, since it implies that $(R_j)\rho^{-1}$ has infinite intersection with $R_i$ and $R_{i'}$. Hence if $(R_i)\rho$ and $R_j$ have infinite intersection (where $i,j \in \Z_{T^n(U)}$), then $(R_i)\rho$ and $R_j$ are almost equal.

Let $\rho \in G$ and let $\omega:=\rho\sigma_\rho^{-1} \in \Sym$ so that $\omega$ sends almost all of each branch of $X_{T^n(U)}$ to itself. Since, for each branch, $\omega$ preserves the number of infinite orbits induced by $g$, we have for all $k \in \Z_{T^n(U)}$ and all $g \in (U)\hat{\phi}_U$ that  $t_k(\omega^{-1}g\omega)=t_k(g)$. Fix an $i \in \Z_{T^n(U)}$ and choose $g \in H_{T^n(U)}$ so that $t_i(g)=1$. Note that $g^{-1}\omega g\omega^{-1} \in \FSym$. Thus there is a $d \in \N$ such that, for all $m>d$, $(i,m)g^{-1}\omega g\omega^{-1}=(i,m)$.  We may now assume for some $m'>d$ that $(i, m')\omega=(i, m'+s)$, where $s\in \N$. This is because $\omega$ sends only finitely many points of $R_i$ to another branch and to ensure the positivity of $s$ we may replace $\omega$ with $\omega^{-1}$. Hence
\[(i,m'+1)g^{-1}\omega g\omega^{-1}=(i,m')\omega g\omega^{-1}=(i,m'+s)g\omega^{-1}=(i,m'+s+1)\omega^{-1}.\]
But, from our assumptions, $(i,m'+1)g^{-1}\omega g\omega^{-1}=(i,m'+1)$. Hence
\[\omega: (i,m'+1)\mapsto (i,m'+s+1)\]
i.e.\ $\omega: (i,m)\mapsto (i,m+s)$ for all $m\ge m'$. Running this argument for each $k \in \Z_{T^n(U)}$ we have  for any $\rho \in G$ that $\rho\sigma_\rho^{-1} \in H_{T^n(U)}$.

We now describe necessary conditions on the branches of $X_{T^n(U)}$ for there to be a $\rho \in G$ that permutes those branches. Let us assume that $\rho \in G$ and that $(R_j)\rho$ is almost equal to $R_{j'}$, where $R_j\subseteq Q_k$ and $R_{j'}\subseteq Q_{k'}$. If $k=k'$ then for all $g \in (U)\hat{\phi}_U$, $t_j(g)=t_{j'}(g)$; hence all  permutations of the branches in $Q_k$ may lie in $G$. If $k\ne k'$ then let $g \in U$ be such that $t_k(g)>0$, $t_{k'}(g)<0$, and $t_i(g)=0$ for all branches $i$ in $X_n\setminus (R_k\cup R_{k'})$. Such an element exists by Lemma \ref{lemfactsaboutU}: there is a $p \in \N$ such that $U_p \le U$. Let $h:=(g)\hat{\phi}$. For a ray $j'$ in $Q_{k'}\subseteq X_{T^n(U)}$ we have that $t_{j'}(\rho^{-1}h\rho)>0$ and so, by (\ref{structureofU}), if $\rho^{-1}h\rho \in (U)\hat{\phi}_U$ then we must have for all branches $i'$ in $Q_{k'}$ that $t_{i'}(\rho^{-1}h\rho)>0$. From our choice of $g$, we may conclude that $Q_k$ cannot contain fewer branches than  $Q_{k'}$. Similarly $t_j(\rho h\rho^{-1})<0$ and so for all rays $i$ in $Q_{k}$, $t_{i}(\rho h\rho^{-1})<0$ meaning that $Q_{k'}$ cannot contain fewer branches than $Q_k$. Hence a necessary condition on $j, j' \in \Z_{T^n(U)}$ for there to be a $\rho \in G$ such that $(R_j)\rho$ is almost equal to $R_{j'}$ is that $Q_k\supseteq R_j$ and $Q_{k'}\supseteq R_{j'}$ contain the same number of branches. This condition is equivalent to the statement that $T_k(U)=T_{k'}(U)$.
\end{proof}
\begin{rem}\label{R} We can more precisely describe the subgroup $R$ from the previous lemma. We have that $R=(\bigoplus_{i=1}^nS_{T_i(U)})\rtimes A$, where $A\le S_n$ corresponds to permuting the factors of the summand that have the same size. More explicitly, if $\sigma\in S_n$ and $\sigma=\sigma_1\ldots\sigma_d$ in disjoint cycle notation, then $\sigma\in A$ if and only if, for every $i\in \Z_d$ and $k, k' \in \Supp(\sigma_i)$, $T_k(U)=T_{k'}(U)$.
\end{rem}
We now consider the particular case of $U_m:=\langle \FAlt(X_n), g_i^m\mid i \in \Z_n\rangle\le_fH_n$ (where $g_i$ is our standard generator of $H_n$ with $t_1(g)=1$ and $t_i(g)=-1$). In this case the sets $Q_i$, $i\in \Z_n$, all consist of $m$ rays. Thus Lemma \ref{factsaboutphihat} states that
\begin{align}\label{showequal}
N_{\Sym(X_{mn})}((U_m)\hat{\phi}_{U_m})\le H_{mn}\rtimes (S_m\wr S_n)
\end{align}
where $S_m\wr S_n=(\bigoplus_{i=1}^nS_m)\rtimes S_n$. In fact (\ref{showequal}) is an equality. Let $\rho \in H_{mn}\rtimes (S_m\wr S_n)$. Then conjugation by $\rho$ preserves cycle type. Thus $\rho^{-1}(\FAlt(X_{mn}))\rho=\FAlt(X_{mn})$. Given $g_i^m \in U_m$, we have that $(g_i^m)\hat{\phi}_{U_m}$ is a product of $m$ infinite cycles, each with support equal to two branches of $X_{mn}$. Since conjugation by any $\omega\in H_{mn}$ preserves cycle type and sends almost all of each branch of $X_{mn}$ to itself, we have that $\omega^{-1}((g_i^m)\hat{\phi}_{U_m})\omega \in (U_m)\hat{\phi}_{U_m}$. Elements of the head of $S_m\wr S_n$ send $(g_i^m)\hat{\phi}_{U_m}$ to $(g_j^mg_k^{-m})\hat{\phi}_{U_m}$ for some $j, k \in \Z_n$. We now consider the preimage, under $\hat{\phi}_{U_m}$, of elements of the base of $S_m\wr S_n$. 

\begin{not*} Let $Y_{i, 0}(U):=\{(i, m)\mid 1\le m\le T_i(U)\}$ and, for any $s \in \N$, let $Y_{i, s}(U):=\{(i,m)\mid sT_i(U)+1\le m\le (s+1)T_i(U)\}=Y_{i, 0}(U)g_i^{-sT_i(U)}$. Thus $R_i=\bigsqcup\limits_{s=0}^\infty Y_{i, s}(U)=\bigsqcup\limits_{s=0}^\infty Y_{i, 0}(U)g_i^{-sT_i(U)}$.
\end{not*}
\begin{defn} Given any  $i \in \Z_n$ and $\sigma \in \FSym(X_n)$ with $\Supp(\sigma)\subseteq Y_{i, 0}(U))$, let $u_{\sigma, i}$ be the element of $\Sym(X_n)$ such that $\Supp(u_{\sigma, i})\subseteq R_i$ and, for every $s \in \N\cup\{0\}$, $u_{\sigma, i}\big|Y_{i, s}(U)=g_i^{sT_i(U)}\sigma g_i^{-sT_i(U)}$ i.e.\ let $u_{\sigma, i}$ induce the permutation $\sigma$ on every set $Y_{i, s}(U)$.
\end{defn}
The preimage of the base of $S_m\wr S_n$ is therefore $\{u_{\sigma, i}\mid i \in \Z_n, \sigma \in \FSym(Y_{i,0}(U_m))\}$. Let $j\in \Z_n$ and $\sigma \in \FSym(Y_{j,0}(U_m))$. We have that conjugation by $u_{\sigma, j}$ sends $g_i^m$ to an element of $H_n$ and, in particular, $t_1(u_{\sigma, j}^{-1}g_i^mu_{\sigma, j})=-t_i(u_{\sigma, j}^{-1}g_i^mu_{\sigma, j})=m$. Now $u_{\sigma, j}^{-1}g_i^mu_{\sigma, j}\in U_m$ since conjugation by elements of $\Sym$ preserves cycle type. Hence all elements of $H_{mn}\rtimes (S_m\wr S_n)$ normalise $(U_m)\hat{\phi}_{U_m}$.

\begin{prop}\label{prop-structureofUp} Let $n\in \{2, 3, \ldots\}$ and $U\le_fH_n$. Then there exists an $m \in \N$ such that $U_m\le U$ and $N_{\Sym(X_n)}(U)\le N_{\Sym(X_n)}(U_m)$. Importantly this implies that $U_m$ is characteristic in $U$.
\end{prop}
\begin{proof}  Let $n\in \{2, 3, \ldots\}$ and $U\le_fH_n$. Lemma \ref{lemfactsaboutU} states that there exists an $m \in 2\N$ such that $U_m\le_fU$. Let $G_m:=N_{\Sym(X_{mn})}((U_m)\hat{\phi}_{U_m})$, which, from the above, equals $H_{mn}\rtimes (S_m\wr S_n)$. We wish to show that $N_{\Sym(X_n)}(U)\le N_{\Sym(X_n)}(U_m)$. From Lemma \ref{factsaboutphihat}, $N_{\Sym(X_{T^n(U)})}((U)\hat{\phi})\le H_{T^n(U)}\rtimes R$ where $R\le S_{T^n(U)}$. It is therefore sufficient to show that $(H_{T^n(U)}\rtimes R)\hat{\phi}_U^{-1}\le (G_m)\hat{\phi}_{U_m}^{-1}$.

First consider $(R)\hat{\phi}_U^{-1}$. By Remark \ref{R}, $R=(\bigoplus_{i=1}^nS_{T_i(U)})\rtimes A$. But $(A)\hat{\phi}_U^{-1}\le S_n$ corresponds to the isometric permutations of the rays of $X_n$, all of which lie in $(G_m)\hat{\phi}_{U_m}^{-1}$. Now $(\bigoplus_{i=1}^nS_{T_i(U)})\hat{\phi}_U^{-1}$ consists of $\{u_{\sigma, i}\mid i \in \Z_n, \sigma \in \FSym(Y_{i,0}(U))\}$. But since $T_i(U)\big|m$ for every $i\in \Z_n$, this is a subset of $\{u_{\sigma, i}\mid i \in \Z_n, \sigma \in \FSym(Y_{i,0}(U_m))\}\le (G_m)\hat{\phi}_{U_m}^{-1}$.

Finally we consider $(H_{T^n(U)})\hat{\phi}_U^{-1}$. Let $g_j\in H_{T^n(U)}$. Then $t_1(g_j)=1$ and $t_j(g_j)=-1$ where $R_j\subseteq Q_k$. We first consider the image of $g_j$ under $\hat{\phi}_U^{-1}$. Since $(R)\hat{\phi}_U^{-1}\le (G_m)\hat{\phi}_{U_m}^{-1}$, we may assume that $j$ is the lowest numbered branch in $Q_k$, meaning that $\Supp((g_j)\hat{\phi}_U^{-1})=(\Supp(g_j))\phi_U^{-1}=X_{1, 1}(g_j^{T_1(U)})\sqcup X_{k, 1}(g_j^{T_k(U)})$. We now consider the image of this element under $\hat{\phi}_{U_m}$. For every $d\in\N$ let $i_d:=(d-1)\cdot T_1(U)+1$, let $k_d:=(d-1)\cdot T_k(U)+(k-1)m+1$ and, for each $i\in \Z_n$, let $c_i$ denote the constant such that $m=c_iT_i(U)$. Then
\begin{align*}
(X_{1, 1}(g_k^{T_1(U)}))\phi_{U_m}=\left(\bigsqcup_{d=1}^{c_1}X_{1, i_d}(g_k^{c_1T_1(U)})\right)\phi_{U_m}=\bigsqcup_{d=1}^{c_1}R_{i_d}
\end{align*}
and
\begin{align*}
(X_{k, 1}(g_k^{T_k(U)}))\phi_{U_m}=\left(\bigsqcup_{d=1}^{c_k}X_{1, k_d}(g_k^{c_kT_k(U)})\right)\phi_{U_m}=\bigsqcup_{d=1}^{c_k}R_{k_d}
\end{align*}
where $c_1T_1(U)=c_kT_k(U)=m$. Thus $\Supp((g_j)\hat{\phi}_U^{-1}\hat{\phi}_{U_m})=(\bigsqcup_{d=1}^{c_1}R_{i_d})\sqcup(\bigsqcup_{d=1}^{c_k}R_{k_d})$. Consider the element $g\in G_m$ that: has one infinite orbit; preserves the colexicographic order (inherited from this ordering on $X_{mn}=\{(i,m)\mid i \in \Z_{mn}, m\in\N\}$) on $\bigsqcup_{d=1}^{c_1}R_{i_d}$ and on $\bigsqcup_{d=1}^{c_k}R_{k_d}$; sends $(k_1, 1)$ to $(1, 1)$; and fixes all other points of $X_{mn}$. We note that $g$ and $(g_j)\hat{\phi}_U^{-1}\hat{\phi}_{U_m}$ have the same supports, and considering their actions on $X_{mn}$ we see that $g=(g_j)\hat{\phi}_U^{-1}\hat{\phi}_{U_m}$.
\end{proof}
\subsection{Conjugacy for groups commensurable to $H_n$}\label{commconj}
In Section \ref{appendixprop} we show that there exists an algorithm which, for any $n \in \{2, 3, \ldots\}$, $p \in \N$, and $H_{np}$-conjugated $a, b \in H_{np}\rtimes S_{np}$, decides whether $a$ and $b$ are $(U_p)\hat{\phi}_{U_p}$-conjugated.
\begin{prop}\label{prop-automorphismgrouplattice} Let $n\in \{2, 3, \ldots\}, p \in 2\N$ and $U_p\le H_n$. Then TCP($U_p$) is solvable.
\end{prop}
\begin{proof} Our aim is to produce an algorithm which, given $a, b \in U_p$ and $\phi_\rho \in \Aut(U_p)$, decides whether there exists a $u \in U_p$ such that $(u^{-1})\phi_\rho au=b$ i.e.\ $u^{-1}\rho au=\rho b$. Let $\hat{\phi}:=\hat{\phi}_{U_p}$ and let us rephrase our question in $(U_p)\hat{\phi}$:
\begin{align*}
u^{-1}\rho au&=\rho b\\
\Leftrightarrow (u^{-1}\rho au)\hat{\phi}&=(\rho b)\hat{\phi}\\
\Leftrightarrow (u^{-1})\hat{\phi}(\rho a)\hat{\phi}(u)\hat{\phi}&=(\rho b)\hat{\phi}
\end{align*}
where $(\rho a)\hat{\phi}, (\rho b)\hat{\phi} \in H_{np}\rtimes S_{np}$  and $(u)\hat{\phi} \in (U_p)\hat{\phi}\le H_{np}$ from Lemma \ref{factsaboutphihat1} and Lemma \ref{factsaboutphihat}. The algorithm for TCP($H_{np}$) in Section \ref{sectionTCPHnalgorithm} may be used to produce a conjugator $x \in H_{np}$ if one exists. Given such a $x$, Proposition \ref{prop-conjinUp} decides whether there exists a $y \in (U_p)\hat{\phi}$ which conjugates $(\rho a)\hat{\phi}$ to $(\rho b)\hat{\phi}$.
\end{proof}
\begin{prop}\label{prop-Uporbitdecide} Let $n\in \{2, 3, \ldots\}$, $p \in 2\N$, and $U_p\le H_n$. If $E$ is a finite extension of $U_p$, then $A_{U_p\unlhd E}$ is orbit decidable.
\end{prop}
\begin{proof}
Recall that for $A_{U_p\unlhd E}=\{\phi_e\mid e \in E\}$ to be orbit decidable, there must exist an algorithm which decides, given any $a', b' \in U_p$, whether there exists a $\psi \in A_{U_p\unlhd E}$ such that
\begin{align}\label{initialeqn}
  (a')\psi= b'.
\end{align}
Since $\Aut(U_p)\cong N_{\Sym(X_n)}(U_p)$, we may rewrite (\ref{initialeqn}) as searching for an element $\phi_\rho \in \Aut(U_p)$ such that
\begin{align*}
(a')\phi_\rho=b'\;\text{and}\;\phi_\rho \in A_{U_p\unlhd E}
\end{align*}
i.e.\ searching for a $\rho \in \bar{E}$, the image of the natural epimorphism from $A_{U_p\unlhd E}$, such that $\rho^{-1}a'\rho=b'$. Now we rephrase this question using the map $\hat{\phi}:=\hat{\phi}_{U_p}$:
\begin{align*}
(\rho^{-1})\hat{\phi}(a')\hat{\phi}(\rho)\hat{\phi}=(b')\hat{\phi}.
\end{align*}
Let $a:=(a')\hat{\phi}$, $b:=(b')\hat{\phi}$, and $y:=(\rho)\hat{\phi}$. Thus $a, b \in (U_p)\hat{\phi}\le H_{np}$ are known, and $y$ must be chosen to be any element in $(\bar{E})\hat{\phi}$ so that $y^{-1}ay=b$. We have that $(U_p)\hat{\phi}\le_c (\bar{E})\hat{\phi}\le H_{np}\rtimes S_{np}$ for some $c\in \N$. Thus there exist coset representatives $e_2, \ldots,e_c$ and $(\bar{E})\hat{\phi}=\langle (g_2^p)\hat{\phi},\ldots, (g_n^p)\hat{\phi}, \FAlt(X_{mn}), e_2, \ldots, e_c\rangle$. Now, deciding whether $a, b$ are $(\bar{E})\hat{\phi}$-conjugated is equivalent to deciding whether any pair in $\{(a, e_ibe_i^{-1})\mid 2\le i\le c\}$ is $(U_p)\hat{\phi}$-conjugated, by noting that every element in $(\bar{E})\hat{\phi}$ decomposes as $ue_i$ for some $u\in (U_p)\hat{\phi}$ and $2\le i\le c$ and that $u^{-1}au=e_ibe_i^{-1}\Leftrightarrow (ue_i)^{-1}a(ue_i)=b$. Deciding if any pair is $(U_p)\hat{\phi}$-conjugated is possible by first deciding whether any pair is $H_{mn}$-conjugated and then, for each such pair, applying Proposition \ref{prop-conjinUp}.
\end{proof}

\begin{prop}\label{prop-commensurableconjugacy} Let $n\in \{2, 3, \ldots\}$, $p\in 2\N$, and $U_p\unlhd_fG$. Then CP($G$) is solvable.
\end{prop}
\begin{proof} We again use \cite[Thm. 3.1]{orbitdecide}. $G$ is a finite extension of $U_p$ by $F$, some finite group. TCP($U_p$) is solvable by Proposition \ref{prop-automorphismgrouplattice}. $A_{U_p\unlhd G}$ is orbit decidable by Proposition \ref{prop-Uporbitdecide}. Hence CP($G$) is solvable.
\end{proof}

Recall that $A$ and $B$ are commensurable if and only if there exist $N_A\cong N_B$ with $N_A$ finite index and normal in $A$ and $N_B$ finite index and normal in $B$. Our aim is to prove Theorem \ref{main5}, that, for any $n\in \{2, 3, \ldots\}$ and any group $G$ commensurable to $H_n$, CP($G$) is solvable.
\begin{proof}[Proof of Theorem \ref{main5}] Fix an $n\in \{2, 3, \ldots\}$ and let $G$ and $H_n$ be commensurable. Then there is a $U\unlhd_f G, H_n$. By Proposition \ref{prop-structureofUp}, there exists an $m \in 2\N$ such that $U_m$ is finite index and characteristic in $U$. It is a well know result that if $A$ is characteristic in $B$ and $B$ is normal in $C$, then $A$ is normal in $C$. Hence $U_m\unlhd G$ and we may apply Proposition \ref{prop-commensurableconjugacy} to obtain that CP($G$) is solvable.
\end{proof}

\section{Further computational results}
\subsection{Computational results regarding centralisers in $H_n$}\label{6.1}
\begin{prop} \label{1} There is an algorithm that, given any $n\in \{2, 3, \ldots\}$ and any $g \in \G$, outputs a finite generating set for $t(C_{H_n}(g))\le \Z^n$.
\end{prop}
In order to prove Proposition \ref{1} we need some notation.
\begin{not*}
Given a subset $A\subseteq \Z_n$ and $g\in\G$, let $t_A(g)=(x_1,\ldots,x_n)\in\Z^n$ be given by $x_i=t_i(g)$ if $i\in A$ and $x_i=0$ otherwise.
\end{not*}
Note that $t_A$ is not a homomorphism: consider $h$, an isometric permutation of the 1st and 2nd branches of $X_n$, and $g_2$, a standard generator of $H_n$. Then $t_1(h^{-1}g_2h)=t_1(g_2^{-1})=-1$ but $t_1(h^{-1})+t_1(g_2)+t_1(h)=0+1+0$. It is a homomorphism from $H_n$ however. Also, given any $g\in \G$ and $A\subseteq \Z_n$, the vector $t_A(g)$ is computable by Lemma \ref{lem-actioncomputable}.
\begin{lem} \label{2} Given $n\in \{2, 3, \ldots\}$ and $g \in \G$,
\[t(C_{H_n}(g))=t_{I^c(g)}(C_{H_n}(g))\oplus t_{I(g)}(C_{H_n}(g)).\]
\end{lem}
\begin{proof} Let $h \in C_{H_n}(g)$ and $i \in I(g)$. From Lemma \ref{lemmacentraliser}, $h_{[i]}\in C_{H_n}(g)$ and $t_j(hh^{-1}_{[i]})=0$ if $j \in \mathfrak{C}_g([i])$ and $t_j(hh^{-1}_{[i]})=t_j(h)$ otherwise. Thus, repeating the argument for all classes of $I(g)$, we construct $h' \in C_{H_n}(g)$ with $t(h')=t_{I^c(g)}(h)$, and thus $t(h(h')^{-1})=t_{I(g)}(h)$. It follows that $t(h)\in \langle t_{I^c(g)}(C_{H_n}(g)), t_{I(g)}(C_{H_n}(g))\rangle$. Clearly $t_{I^c(g)}(C_{H_n}(g))\cap t_{I(g)}(C_{H_n}(g))$ is the trivial element of $\Z^n$ and hence they generate a direct sum.
\end{proof}

\begin{lem} \label{3} Let $n\in \{2, 3, \ldots\}$, $g \in \G$, and $r\in \N$. Then $\sum_{i\in I^c_r(g)}t_i(h)=0$ for all $h \in C_{H_n}(g)$.
\end{lem}
\begin{proof} By Lemma \ref{constantsA}, $t_j(h)=t_i(h)$ for all $j \in [i]_g$. If $j_r^1, \ldots j_r^u$ are representatives of $I^c_r(g)$, then $\sum_{i\in I^c_r(g)}t_i(h)=r\sum_{s=1}^ut_{j_r^s}(h)=r\sum_{s=1}^ut_{j_r^s}(\id)=0$ by Lemma \ref{lem-sametransonIc}.
\end{proof}
\begin{defn}\label{settheta} Given $g \in \G$, let $\Theta(g)$ be the set of all $c_{j_r^d, j_r^{d'}}$ of the statement of Lemma \ref{introducecjj} where $r \in \N$, $j_r^1, \ldots , j_r^u$ are distinct representatives of $I_r^c(g)$, and $d, d' \in \Z_u$ are distinct. Note that since $I^c(g)$ is finite, $\Theta(g)$ is finite.
\end{defn}
\begin{lem} \label{4} Given $n \in \{2, 3, \ldots\}$ and $g \in  \G$, $t_{I^c(g)}(C_{H_n}(g))$ is generated by the image of $\Theta(g)$ under $t$.
\end{lem}
\begin{proof} By Lemma \ref{constantsA} and Lemma \ref{3}, for any $h \in  C_{H_n} (g)$ and for every $r \in  \N$ we have that $\sum_{i\in I_r^c(g)} t_i(h) = 0$ and $t_i(h) = t_k(h)$ if $[k] = [i]$. Thus $t_{I^c(g)}(C_{H_n}(g))$ must be contained in
\[A:=\{(x_1,\ldots,x_n)\in \Z^n \mid x_i =0\;\text{if}\;i\in I(g), \sum_{i\in I_r^c(g)}\!\!\!x_i =0\;\text{for all}\;r\in \N,x_j =x_k\;\text{if}\;[j]=[k]\}.\]
Note that it follows from the definition of $\Theta(g)$ that $t(c) \in  t_{I^c(g)}(C_{H_n} (g))$ for all $c \in  \Theta(g)$. Thus $\langle t(\Theta(g))\rangle \subseteq t_{I^c(g)}(C_{H_n} (g)) \subseteq A$. We claim that $\langle t(\Theta(g))\rangle = A$. Clearly, the lemma follows now from the claim. To see that $t(\Theta(g))$ generates $A$, we argue by induction on the $L_1$-norm of $(x_1,\ldots,x_n) \in  A$ (i.e. $||(x_1,...,x_n)|| = \sum|x_i|)$. Let $y \in  A$. If $||y|| = 0$, then $y \in  \langle t(\Theta(g))\rangle$. Suppose that $y = (x_1,...,x_n)$. If $||y|| > 0$, since $x_i = 0$ for all $i\in I(g)$, there must be $j\in I^c$ with $x_j\ne0$. Assume $x_j <0$. Let $r\in \N$ such that $j\in I_r^c(g)$. Then, by Lemma \ref{3}, there must be $k \in  I_r^c(g)$, $[k]\ne [j]$ with $x_k > 0$. Observe that $t_l(c_{j,k})$ is equal to 1 for $l \in [j]$, is equal to $-1$ for $l \in  [k]$, and zero otherwise. Hence $||yt(c_{k,j})|| = ||y||-2r$, and we conclude by induction that $yt(c_{k,j}) \in  \langle t(\Theta(g))\rangle$ and hence $y \in  \langle t(\Theta(g))\rangle$.
\end{proof}
We now have to find a generating set for $t_{I(g)}(C_{H_n} (g))$.
\begin{lem} \label{6} Let $n\in \{2, 3, \ldots\}$ and $g \in \G$. Then for each $i \in  I(g)$, $t_{\mathfrak{C}_g([i])}(C_{H_n}(g))\le\Z^n$ is infinite cyclic and
\[t_{I(g)}(C_{H_n}(g))=\bigoplus_{i\in R} t_{\mathfrak{C}_g([i])}(C_{H_n}(g)),\]
where $R$ is a set of different representatives of $I(g)/\sim_g$.
\end{lem}
\begin{proof} Let $i \in  I(g)$. Write $g = \omega_g\sigma_g$. Note that $g^{|\sigma_g|} \in  C_{H_n}(g)$ and that $t((g^{|\sigma_g|})_{[i]})$ is non-zero and lies in $t_{\mathfrak{C}_g([i])}(C_{H_n}(g))$. Thus $t_{\mathfrak{C}_g([i])}(C_{H_n}(g))$ is non-trivial.

Let $h \in  C_{H_n}(g)$, such that $t_i(h)$ is positive and minimum. Take $c \in  C_{H_n}(g)$. We will show that $t_{\mathfrak{C}_g([i])}(c) \in \langle t_{\mathfrak{C}_g([i])}(h)\rangle$.

If $t_i(c) = 0$ then, by Lemma \ref{5}, $t_{\mathfrak{C}_g([i])}(c)$ has all its coordinates equal to zero and hence lies in $\langle t_{\mathfrak{C}_g([i])}(h)\rangle$.

Assume that $t_i(c) > 0$, the other is analogous. By the Euclidean algorithm, $t_i(c) = dt_i(h) + r$ with $0 \le r < t_i(h)$. Note that $r = t_i(ch^{-d})$ and $ch^{-d} \in  C_{H_n}(g)$ and thus $r = 0$ by the minimality of $h$. Therefore $t_i(ch^{-d}) = 0$, and, by Lemma \ref{5}, $t_{\mathfrak{C}_g([i])}(ch^{-d})$ has all its coordinates equal to zero, and hence $t_{\mathfrak{C}_g([i])}(c) = t_{\mathfrak{C}_g([i])}(h^d)$.

By Lemma \ref{makedefnwork} and Lemma \ref{lemmacentraliser}, if $h \in  C_{H_n}(g)$, then $h_{[i]} \in  C_{H_n}(g)$ and $t_{\mathfrak{C}_g([i])}(h) = t_{\mathfrak{C}_g([i])}(h_{[i]}) = t(h_{[i]})$. Thus $t_{\mathfrak{C}_g([i])}(C_{H_n} (g))\le t_{I(g)}(C_{H_n} (g))$. Since $t_{I(g)}(h) = \sum_{j\in R}t(h_{[j]})$, we get that $t_{I(g)}(C_{H_n}(g))$ is generated by $\{t_{\mathfrak{C}_g([j])}(C_{H_n}(g)) \mid j \in  R\}$ and clearly if $[j]\not\sim_g [k]$ the cyclic subgroups $t_{\mathfrak{C}_g([j])}(C_{H_n} (g))$ and $t_{\mathfrak{C}_g([k])}(C_{H_n} (g))$ intersect trivially. Now the direct sum follows.
\end{proof}

\begin{lem} \label{7} There is an algorithm that given any $n \in \{2, 3, \ldots\}$, $g \in  \G$, and $i \in  \Z_n$, decides if $i \in  I(g)$ and moreover, if $i \in  I(g)$, it outputs some $h \in  C_{H_n}(g)$ such that $t(h)$ generates $t_{\mathfrak{C}_g([i])}(C_{H_n} (g))$
\end{lem}
\begin{proof} Deciding if $i \in  I(g)$ follows from Lemma \ref{lem-actioncomputable}.

So suppose that $i \in  I(g)$. By Lemma  \ref{6}, $t_{\mathfrak{C}_g([i])}(C_{H_n}(g))$ is infinite cyclic, so there must be $h \in  C_{H_n} (g)$ such that $t_{\mathfrak{C}_g([i])}(h)$ generates $t_{\mathfrak{C}_g([i])}(C_{H_n} (g))$.

From Lemma \ref{lemmacentraliser} if $h \in  C_{H_n}(g)$ then $h_{[i]} \in  C_{H_n} (g)$ and $t_{\mathfrak{C}_g([i])}(h) = t(h_{[i]})$. Thus we can assume that $h \in  C_{H_n} (g)$ has only infinite orbits, its support has finite intersection with the rays $R_j$, $j\not\in  \mathfrak{C}_g([i])$, and $t(h)$ generates $t_{\mathfrak{C}_g([i])}(C_{H_n} (g))$.

Write $g = \omega_g\sigma_g$. Note that $g^{|\sigma_g|} \in  C_{H_n}(g)$ and that $t((g^{|\sigma_g|})_{[i]})$ is non-zero and lies in $t_{\mathfrak{C}_g([i])}(C_{H_n} (g)) = \langle t(h)\rangle$. Let $g^\ast$ denote $(g^{|\sigma_g|})_{[i]}$ and observe that $h \in  C_{H_n} (g)\Rightarrow h\in C_{H_n} (g^\ast)$. Thus $|t_j(h)| \le |t_j(g^{\ast})|$ for $j \in  [i]$, and $t_j(g^\ast)$ is computable from $g$ by Lemma \ref{lem-actioncomputable}.

We will find computable bounds for $z_k(h)$, $k \in  \Z_n$. If $k \in  I^c(g)$, then Lemma \ref{makedefnwork} states that $\Supp(h)=\Supp(g^\ast)$ and so $z_k(h) \le z_k(g^\ast)$ for all $k \in  I^c(g)$. Now consider if $k \in I(g)$. We will show that $z_k(h) \le z_k(g^\ast) + |t_k(g^\ast)|$. Assume, for a contradiction, that $z_k(h)$ is minimal and that $z_k(h)>z_k(g^\ast) + |t_k(g^\ast)|$.  Since $h\in C_{H_n}(g^\ast)$, we have
\begin{align}\label{eqn-centraliserdefn}
(i, m)g^\ast h=(i, m)hg^\ast\;\textrm{for all}\;(i,m) \in X_n.
\end{align}
We may assume that $t_{k}(g^\ast)<0$, since replacing $g^\ast$ with $(g^\ast)^{-1}$ yields a proof for when $t_{k}(g^\ast)>0$. Let $m>z_k(g^\ast)+|t_k(h)|$. Note that
\[(k,m)g^\ast h=(k, m+t_{k}(g^\ast))h\]
and also that
\[(k,m)g^\ast h=(k,m)hg^\ast=(k, m+t_k(h))g^\ast=(k, m+t_{k}(g^\ast)+t_k(h)).\]
Thus $(k, m+t_{k}(g^\ast))h=(k, m+t_{k}(g^\ast)+t_k(h))$, which, since $t_{k}(g^\ast)<0$, contradicts the minimality of $z_k(h)$.

Let $S = \{s \in  H_n : z_k(s) \le z_k(g) + |t_k(g^\ast)|, |t_k(s)| \le |t_k(g^\ast)|, k \in  \Z_n\}$. Note that $S$ must contain $h \in  C_{H_n} (g)$. Note also that $S$ is finite, and with the given restrictions, one can enumerate all elements of $S$. For each $s \in  S$, using the word problem for $\G$ we can decide if $s \in  C_{H_n} (g)$, and using that the functions $t_A$ are computable, we can decide if $t(s)$ is non-zero and equal to $t_{\mathfrak{C}_g([i])}(s)$. Let $S'$ be the subset of $S$ of the elements satisfying the two conditions above. Finally, since $S'$ is finite, we can find $c \in  S'$ such that $t_{\mathfrak{C}_g([i])}(s) \in  \langle t(c)\rangle$ for all $s \in  S'$. Note that since $h \in  S'$ such an $c$ must exist.
\end{proof}

\begin{proof}[Proof of Proposition \ref{1}] From Lemma \ref{lem-actioncomputable} there is an algorithm that computes $I(g)$, $I^c(g)$, and $\mathfrak{C}_g([i])$ for each $i \in  I(g)$. From Lemma \ref{2} $t(C_{H_n} (g)) = t_{I^c(g)}(C_{H_n} (g)) \oplus t_{I(g)}(C_{H_n} (g))$, and we only need algorithms that gives generating sets for each of the direct summands. By Lemma \ref{4}, there is an algorithm that outputs a generating set for $t_{I^c(g)}(C_{H_n}(g))$. By Lemma \ref{7}, there is an algorithm that computes a generator of $t_{\mathfrak{C}_g([i])}(C_{H_n}(g))$ for each $i \in  I(g)$ and by Lemma \ref{6}, those elements generate $t_{I(g)}(C_{H_n} (g))$.
\end{proof}

\subsection{Deciding conjugacy in $(U_p)\hat{\phi}_{U_p}$}\label{appendixprop}
Recall $U_p:=\langle g_2^p,\ldots, g_n^p, \FAlt(X_n)\rangle\le H_n$. From Section, \ref{secficonj} $(U_p)\hat{\phi}_{U_p}\le H_{np}$ and
\[t((U_p)\hat{\phi}_{U_p})=\left\{(\underbrace{a_1\ldots a_1}_p\underbrace{a_2\ldots a_2}_p\ldots\underbrace{a_{n}\ldots a_{n}}_p)^T\,\middle|\,a_1, \ldots, a_{n-1} \in \Z, a_n=-\sum_{i=1}^{n-1}a_i\right\}.\]
When $p$ is even, $\FSym(X_{np})\cap (U_p)\hat{\phi}_{U_p}=\FAlt(X_{np})$. In this section we prove the following (which has been used in Section \ref{commconj} above).
\begin{prop}\label{prop-conjinUp} There is an algorithm which, given any $n\in \{2, 3, \ldots\}$, $p \in2\N$, $a, b \in H_{np}\rtimes S_{np}$ and an $x \in H_{np}$ which conjugates $a$ to $b$, decides whether $a$ and $b$ are $(U_p)\hat{\phi}_{U_p}$-conjugated.
\end{prop}
Lemma \ref{lemcentraliserlemma} stated that if $x\in G$ conjugates $a$ to $b$ then $y\in G$ also conjugates $a$ to $b$ if and only if there exists a $c \in C_{G}(a)$ such that $cx=y$.
\begin{lem}\label{searchforcentraliser} Let $a, b\in H_{np}\rtimes S_{np}$ be conjugate by $x \in H_{np}$. Then $a$ and $b$ are $(U_p)\hat{\phi}_{U_p}$-conjugated if and only if there is a $c \in C_{H_{np}}(a)$ such that $cx \in (U_p)\hat{\phi}_{U_p}$.
\end{lem}
\begin{proof} We apply Lemma \ref{lemcentraliserlemma}. If there is a $c \in C_{H_{np}}(a)$ such that $cx \in (U_p)\hat{\phi}_{U_p}$ then $a$ and $b$ are conjugate by $cx \in (U_p)\hat{\phi}_{U_p}$. If there exists $y \in (U_p)\hat{\phi}_{U_p}$ such that $y^{-1}ay=b$, let $c:=yx^{-1}$. Then $c \in C_{H_{np}}(a)$ and $cx=y \in (U_p)\hat{\phi}_{U_p}$.
\end{proof}
\begin{lem}\label{1.2} Let $a, b \in H_{np}\rtimes S_{np}$ and let $x \in H_{np}$ conjugate $a$ to $b$. Then there is an algorithm that decides whether there exists $c\in C_{H_{np}}(a)$ such that $t(cx)\in t((U_p)\hat{\phi}_{U_p})$, and outputs such an element if one exists.
\end{lem}
\begin{proof}
Let $\{\Vec{\delta}_1, \ldots, \Vec{\delta}_e\}$ denote a finite generating set of $t(C_{H_{np}}(a))$, which is computable by Proposition \ref{1}. Deciding whether there is a $c \in C_{H_{np}}(a)$ such that $t(cx) \in t((U_p)\hat{\phi}_{U_p})$ is equivalent to finding powers $\alpha_i$ of the generators $\Vec{\delta}_i \in \Z^n$ such that
\begin{align*}
t(x)+\sum\limits_{i=1}^e\alpha_i\Vec{\delta}_i\in t((U_p)\hat{\phi}_{U_p}).
\end{align*}
Hence we must decide whether there are constants $\{a_1, \ldots, a_{n-1}\}$ and $\{\alpha_1, \ldots, \alpha_e\}$ such that
\begin{align}\label{eqntobesolved2}
t(x)+\sum\limits_{i=1}^e\alpha_i\Vec{\delta}_i=(a_1\ldots a_1\;a_2\ldots a_2 \ldots\ldots a_{n} \ldots a_{n})^T,\;\text{where}\;a_n:=-\sum\limits_{i=1}^{n-1}a_i.
\end{align}
Viewing this as $np$ linear equations, this system of equations has a solution if and only if there is an element $c \in C_{H_{np}}(a)$ such that $t(cx) \in t((U_p)\hat{\phi}_{U_p})$. By writing these equations as a matrix equation we may compute the Smith normal form to decide if the equations have an integer solution (see, for example, \cite{linearequations}) and compute one should one exist (either from the reference or by enumerating all possible inputs). Definition \ref{settheta} and Lemma \ref{7} then provide suitable preimages for the elements $\{\Vec{\delta}_1, \ldots, \Vec{\delta}_e\}$ in order to output a viable $c$.
\end{proof}
\begin{lem} \label{1.3} Let $n\in \{2, 3, \ldots\}$ and $g \in \G$. Then it is algorithmically decidable whether $C_{H_n}(g)\cap (\FSym(X_n)\setminus\FAlt(X_n))$ is empty. A necessary condition for it to be empty is that $I^c(g)=\emptyset$.
\end{lem}
\begin{proof}
We use Lemma \ref{lem-actioncomputable} to compute $I^c(g)$. If there exists $j\in I^c(g)$, either $|[j]_g|$ is odd or even. If $|[j]_g|$ is even, then we have that the first $|[j]_g|$-cycle on this branch is a finite order element of the centraliser that lies in $\FSym(X_n)\setminus \FAlt(X_n)$. If $|[j]_g|$ is odd, then the element that permutes only the first two $|[j]_g|$-cycles of the branches $[j]_g$ provides an element of the centraliser in $\FSym(X_n)\setminus \FAlt(X_n)$. Hence $C_{H_n}(g)\cap \big(\FSym(X_n)\setminus \FAlt(X_n)\big)$ is non-empty whenever $g$ has an even length orbit or $g$ has two or more orbits of length $m$ for some odd number $m$.  Since one of these is satisfied if $I^c(g)$ is non-empty, assume that $I(g)=\Z_n$. From Section \ref{sectionorbits}, this implies that all of the finite orbits of $g$ lie inside a finite subset of $X_n$, which is computable by Lemma \ref{lem-actioncomputable}. Searching for orbits of $g$ of the appropriate length is therefore also possible in this case.
\end{proof}

\begin{defn} Let $\textrm{sgn}_X$: $\FSym(X)\rightarrow \{1, -1\}$ be the sign function for $\FSym(X)$, so that the preimage of $1$ is $\FAlt(X)$. For any $g\in H_n$, let $f_g:=g(\prod_{i=2}^ng_i^{t_i(g)})$. Note that $t(f_g)=\Vec{0}$ and so $f_g \in \FSym(X_n)$. Thus let $\xi_n: H_n\rightarrow \{1, -1\}$, $h\mapsto \textrm{sgn}_{X_n}(f_h)$.
\end{defn}
Note that for any $g\in H_n$, $\xi_n(g)=\xi_n(f_g)$.
\begin{lem}\label{1.5} There exists an algorithm which, given $g\in H_n$, computes $\xi_n(g)$.
\end{lem}
\begin{proof} Since $g(\prod_{i=2}^ng_i^{t_i(g)})$ lies in $\FSym(X_n)$, we may apply Lemma \ref{lem-actioncomputable} to determine its cycle type, which is sufficient to compute $\xi_n(g)$.
\end{proof}
The previous definition exactly captures the form of the elements in $(U_p)\hat{\phi}_{U_p}$.
\begin{lem}\label{1.6} Let $n\in \{2, 3, \ldots\}$, $p\in 2\N$, $g\in H_{np}$, and $t(g)\in t((U_p)\hat{\phi}_{U_p})$. Then $g \in (U_p)\hat{\phi}_{U_p}$ if and only if $\xi_{np}(g)=1$.
\end{lem}
\begin{proof} Let $h:=(g)\hat{\phi}_{U_p}^{-1}$. Using that $t(g)\in t((U_p)\hat{\phi}_{U_p})$ and that $p\big|t_i(h)$ for every $i\in \Z_{n}$: $h\in U_p\Leftrightarrow f_h \in U_p \Leftrightarrow f_h \in \FAlt(X_{n})\Leftrightarrow f_g \in \FAlt(X_{np})\Leftrightarrow \xi_{np}(g)=1$.

\end{proof}
\begin{lem}\label{1.7} Let $n\in \{2, 3, \ldots\}$. Then $\xi_n$ is a homomorphism.
\end{lem}
\begin{proof}
Clearly the function $\textrm{sgn}_X$ is a homomorphism. Let $g, h \in H_n$. Then
\[gh=f_g(\prod_{i=2}^ng_i^{-t_i(g)})f_h(\prod_{i=2}^ng_i^{-t_i(h)})=f_gf_h'(\prod_{i=2}^ng_i^{-(t_i(g)+t_i(h))})\]
where $f_h'=(\prod_{i=2}^ng_i^{-t_i(g)})f_h(\prod_{i=2}^ng_i^{t_i(g)})$ has the same cycle type as $f_h$. Thus $f_{gh}=f_gf_h'$ and $\xi_n(gh)=\xi_n(f_{gh})=\xi_n(f_gf_h')=\xi_n(f_g)\xi_n(f_h')=\xi_n(f_g)\xi_n(f_h)=\xi_n(g)\xi_n(h)$. 
\end{proof}
\begin{proof}[Proof of Proposition \ref{prop-conjinUp}] By Lemma \ref{searchforcentraliser}, $a$ and $b$ are $(U_p)\hat{\phi}_{U_p}$-conjugated if and only if there exists a $c\in C_{H_{np}}(a)$ such that $cx\in (U_p)\hat{\phi}_{U_p}$. A necessary condition for $cx\in (U_p)\hat{\phi}_{U_p}$ is that $t(cx)\in (U_p)\hat{\phi}_{U_p}$.  Lemma \ref{1.2} decides whether this condition can be satisfied. If not, then $C_{H_{np}}(a)x\cap (U_p)\hat{\phi}_{U_p}$ is empty and $a$ and $b$ are not $(U_p)\hat{\phi}_{U_p}$-conjugated. We can therefore assume that this condition is satisfied.

Let $c \in C_{H_{np}}(a)$ be the element outputted by the algorithm of Lemma \ref{1.2} so that $t(cx)\in (U_p)\hat{\phi}_{U_p}$. By Lemma \ref{1.6}, $cx\in (U_p)\hat{\phi}_{U_p}$ if and only if $\xi_{np}(cx)=1$. By Lemma \ref{1.5}, $\xi_{np}(cx)$ is computable. If $\xi_{np}(cx)=1$ we are done and so let us assume that $\xi_{np}(cx)=-1$. If there exists $c'\in C_{H_{np}}(a)\cap (\FSym(X_{np})\setminus\FAlt(X_{np}))$ then $c'c\in C_{H_{np}}(a)$, $t(c'cx)=t(c')+t(cx)=t(cx)\in t((U_p)\hat{\phi}_{U_p})$, and $\xi_{np}(c'cx)=\xi_{np}(c')\xi_{np}(cx)=1$ since $\xi_{np}$ is a homomorphism by Lemma \ref{1.7}. Hence $c'cx$ conjugates $a$ to $b$ and $c'cx \in(U_p)\hat{\phi}_{U_p}$. Lemma \ref{1.3} decides whether such a $c'$ exists. If no such $c'$ exists, Lemma \ref{1.3} states that $I^c=\emptyset$. Thus the generating set $\{\Vec{\delta}_1, \ldots, \Vec{\delta}_e\}$ for $t(C_{H_{np}}(a))$ from Proposition \ref{1} consists only of those elements from Lemma \ref{7}. Lemma \ref{7} computes $\hat{\Vec{\delta}}_1, \ldots, \hat{\Vec{\delta}}_e \in C_{H_{np}}(a)$ such that $t(\hat{\Vec{\delta}}_j)=\Vec{\delta}_j$ for each $j \in \Z_e$. We then take the equations
\begin{align*}
t(x)+\sum\limits_{i=1}^e\alpha_i\Vec{\delta}_i=(a_1\ldots a_1\;a_2\ldots a_2 \ldots\ldots a_{n} \ldots a_{n})^T,\;\text{where}\;a_n:=-\sum\limits_{i=1}^{n-1}a_i.
\end{align*}
labelled (\ref{eqntobesolved2}) above (where $a_1, \ldots, a_{n-1}$ and $\alpha_1, \ldots, a_e$ are to be found) and add the equation
\begin{align*}
 \left(\prod\limits_{j=1}^e\xi(\hat{\Vec{\delta}}_j)^{\alpha_j}\right)\xi(x)=1
\end{align*}
which ensures that the chosen $c\in C_{H_{np}}(a)$ satisfies $\xi(cx)=1$ (from Lemma \ref{1.7}). From our assumptions, the choice of generating set and preimage are arbitrary: if $h, h' \in C_{H_{np}}(a)$ satisfy $t(h)=t(h')$, then we must have that $h^{-1}h' \in \FAlt(X_{np})$, since otherwise $C_{H_{np}}(a)\cap \big(\FSym(X_{np})\setminus \FAlt(X_{np})\big)$ would be non-empty. We may then write these equations as a matrix equation and compute the Smith normal form to decide whether or not the equations have an integer solution.
\end{proof}
\subsection{Proving Theorem 3} The structure of centralisers for elements in $H_n$ was studied in \cite{simon}. In this section we show, for any $g\in\G$, when $C_{H_n}(g)$ is finitely generated and when it is, that a finite generating set is algorithmically computable. From Section \ref{6.1}, a finite set $\{\Vec{\delta}_1, \ldots, \Vec{\delta}_e\}$ is computable, from only $n\in \{2,3,\ldots\}$ and $a\in \G$, such that $t(C_{H_n}(a))=\langle \Vec{\delta}_1, \ldots, \Vec{\delta}_e\rangle$. Hence, given $c \in C_{H_n}(a)$, there exist $\alpha_1, \ldots, \alpha_e \in \Z$ such that $t(c)=\sum_{i=1}^e\alpha_i\Vec{\delta}_i$. 

\begin{not*} Let $\{\Vec{\delta}_1, \ldots, \Vec{\delta}_{e'}\}$ be the image of $\Theta(a)$ under $t$ (where the set $\Theta(a)$ was introduced in Definition \ref{settheta}) and let $\{\Vec{\delta}_{e'+1}, \ldots, \Vec{\delta}_e\}$ be the image under $t$ of the elements outputted by the algorithm of Lemma \ref{7}.
\end{not*}
We will begin to describe a finite generating set for $C_{H_n}(a)$ (should one exist) by choosing certain preimages (under $t$) of $\{\Vec{\delta}_1, \ldots, \Vec{\delta}_e\}$. Lemma \ref{introducecjj} states that $\Theta(a)\subseteq C_{H_n}(a)$, and so our preimages of $\{\Vec{\delta}_1, \ldots, \Vec{\delta}_{e'}\}$ will be the elements of $\Theta(a)$. The following lemma shows that, by imposing the condition that the preimage may only have infinite orbits, there is a unique preimage for the elements $\{\Vec{\delta}_{e'+1}, \ldots, \Vec{\delta}_e\}$.

\begin{lem} Let $n\in \{2, 3,\ldots\}$, $a\in \G$, $i\in I$, and $g\in C_{H_n}(a)$ such that $\langle t(g)\rangle=t_{\mathfrak{C}_a([i])}(C_{H_n}(a))$. If $h\in C_{H_n}(a)$ and $t(h)=t(g)$, then $g_{[i]}=h_{[i]}$. In addition, for each $s\in\{e'+1,\ldots, e\}$ there is a unique preimage $\delta_s$ of $\Vec{\delta}_s$ having only infinite orbits.
\end{lem}
\begin{proof} Lemma \ref{makedefnwork} states that $\Supp(g_{[i]})=\Supp((a^{|\sigma_a|})_{[i]})=\Supp(h_{[i]})$ and so $\Supp(g_{[i]}h_{[i]}^{-1})\subseteq \Supp((a^{|\sigma_a|})_{[i]})$. Now, $g_{[i]}h_{[i]}^{-1}\in\FSym(X_n)$ must be trivial by the first part of the statement of Lemma \ref{makedefnwork}: elements of $C_{H_n}(a)$ cannot have finite orbits intersecting $\Supp((a^{|\sigma_a|})_{[i]})$.
\end{proof}
The previous lemma therefore provides a unique preimage for each of the elements $\{\Vec{\delta}_{e'+1}, \ldots, \Vec{\delta}_e\}$. We now describe the structure of $C_{\FSym(X_n)}(a)$ in order to determine when $C_{H_n}(a)$  is not finitely generated, and to describe a finite generating set when one exists.

\begin{not*} Let $g\in \Sym(X)$ for some non-empty set $X$. Then $Y_g^{(0)}:=\Supp(g\big|_\infty)$, $Y_g^{(1)}:=X\setminus \Supp(g)$, and, for each $r\in \N$, $Y_g^{(r)}:=\Supp(g\big|_r)$. Thus: $Y_g^{(1)}$ denotes the fixed points of $g$; $Y_g^{(r)}$, where $r\in\{2, 3, \ldots\}$, is the union of the orbits of $g$ of size $r$; and $Y_g^{(0)}$ is the union of the infinite orbits of $g$.
\end{not*}

\begin{lem}\label{centralisereltsplitting} Let $g\in \Sym(X)$ and $G\le \Sym(X)$ for some non-empty set $X$. If $c\in C_G(g)$, then $c=\prod_{r=0}^\infty\alpha_r$ where, for every $r\in \{0,1,2,\ldots\}$, $\Supp(\alpha_r)\subseteq Y_g^{(r)}$.
\end{lem}
\begin{proof} We show that $c$ must restrict to a bijection, for each $r\in \{0, 1, 2,\ldots\}$, on the set $Y_g^{(r)}$. Consider if $c$ sent an $r$-cycle of $g$ to an $s$-cycle of $g$ where $r\ne s$. By possibly replacing $c$ with $c^{-1}$, we may assume that $r<s$ (where $s$ may be finite or infinite). Let $y$ be in this orbit of size $r$. Then $(y)g^r=y$. Since $c\in C_G(g)$, $(y)c^{-1}g^rc=y$. But then $c$ sends both $(y)c^{-1}$ and $(y)c^{-1}g^r$ to $y$, a contradiction.
\end{proof}

\begin{lem}\label{thm3mainstructureresult}  Let $g\in \Sym(X)$ for some non-empty set $X$. Then
\[C_{\FSym(X)}(g)\cong \bigoplus\limits_{r=0}^\infty C_{\FSym(Y_g^{(r)})}(g)\]
\end{lem}
\begin{proof} Given $c\in C_{\FSym(X)}(g)$, the previous lemma states that $c=\prod_{r=0}^\infty\alpha_r$ where, for every $r\in \{0,1,2,\ldots\}$, $\Supp(\alpha_r)\subseteq Y_g^{(r)}$. Now $c\in \FSym\Rightarrow |\Supp(c)|<\infty$, and so, for every $r$, $|\Supp(\alpha_r)|<\infty \Rightarrow \alpha_r\in\FSym(Y_g^{(r)})$. Clearly, if $r\ne s$, then $\FSym(Y_g^{(r)})\cap \FSym(Y_g^{(s)})$ is trivial.
\end{proof}

\begin{lem}\label{onepointdefines} Let $g\in \Sym(X)$, where $X$ is a non-empty set, and let $y, z\in X$. If $c \in C_{\Sym(X)}(g)$ and $c:y\mapsto z$, then $c:yg^d\mapsto zg^d$ for all $d\in \Z$.
\end{lem}
\begin{proof} For all $d\in \Z$,
\begin{align*}
(z)c^{-1}g^dc=yg^dc\;\text{and}\;(z)c^{-1}g^dc=(z)g^d.
\end{align*}
Thus $c:yg^d\mapsto zg^d$ for all $d\in \Z$.
\end{proof}

\begin{lem}\label{fsymcentraliserstructure} Let $r\in \N$ and $g\in \Sym(X)$ where $X$ is a non-empty set. Then
\[C_{\FSym(Y_g^{(r)})}(g)\cong C_r\wr\FSym\left(\quotient{Y_g^{(r)}}{g}\right)\]
where $\quotient{Y_g^{(r)}}{g}$ is the set obtained from $Y_g^{(r)}$ by the equivalence relation $x\sim y \Leftrightarrow xg^d=y$ for some $d\in \Z$. Moreover, as a subgroup of $\FSym(Y_g^{(r)})$, the base of this wreath product consists of all of the $r$-cycles of $g$, and the head consists of all finitary permutations of the orbits of $g$ of size $r$.
\end{lem}

\begin{proof}
Let $c\in C_{\FSym(Y_g^{(r)})}(g)$. From Lemma \ref{thm3mainstructureresult} and Lemma \ref{onepointdefines}, $c$ restricts to a bijection of the $r$-cycles of $g$. We will impose the following condition on $c$: whenever $c$ sends an $r$-cycle $\sigma$ of $g$ to an $r$-cycle $\omega$ of $g$, it sends the smallest point of $\Supp(\sigma)$ (under some total ordering on $X$) to the smallest point of $\Supp(\omega)$. We may do this by replacing $c$ with $hc$ where $h$ is a product of $r$-cycles of $g$, since each $r$-cycle of $g$ clearly lies in $C_{\FSym(Y_g^{(r)})}(g)$. Then $c$ is determined by its action on the points in $\quotient{Y_g^{(r)}}{g}$, where the representatives can be chosen to be the smallest point in each $g$-orbit of size $r$. Thus every element of $C_{\FSym(Y_g^{(r)})}(g)$ is a product of $r$-cycles of $g$ together with a permutation of the $r$-cycles of $g$ i.e. $C_{\FSym(Y_g^{(r)})}(g)$ is the above wreath product.
\end{proof}

We now work through the cases $|I_r^c(a)|=0$, $|I_r^c(a)|=r$, and $|I_r^c(a)|\ge 2r$.

\begin{defn}\label{onlyfinite} Let $n\in \{2,3,\ldots\}$ and $g\in \G$. For each $r\in \N$ where $Y_g^{(r)}$ is finite, let $\Omega_r(g)$ consist of the elements of $C_{\FSym(Y_g^{(r)})}(g)$. This set is enumerable since, from our description of orbits in Section \ref{sectionorbits}, all such $r$-cycles of $g$ lie in $Z(g)$ which is finite and computable from only the element $g$.
\end{defn}

\begin{not*} Given $n\in \N$ and $Y\subseteq X_n$, let $H_n(Y):=\{h\in H_n\mid \Supp(h)\subseteq Y\}$.
\end{not*}

\begin{lem}\label{7.4} Let $r\in \N$ and $|I_r^c(g)|= r$. Then $C_{H_n}(g)$ is not finitely generated.
\end{lem}
\begin{proof}
By Lemma \ref{4}, if $c\in C_{H_n}(g)$ then $t_j(c)=0$ for all $j\in I_r^c(g)$. By Lemma \ref{centralisereltsplitting}, if $c\in C_{H_n}(g)$ then $c=fc'$ where $\Supp(f)\subseteq Y_g^{(r)}$ and $\Supp(c')\subseteq X_n\setminus Y_g^{(r)}$. Thus
\[C_{H_n}(g)=C_{\FSym(Y_g^{(r)})}(g)\oplus  C_{H_n(X_n\setminus Y_g^{(r)})}(g).\]

Lemma \ref{fsymcentraliserstructure} states that
\begin{align}\label{eqn-fsymcentraliserstructure}
C_{\FSym(Y_g^{(r)})}(g)\cong C_r\wr\FSym\left(\quotient{Y_g^{(r)}}{g}\right).
\end{align}
If (\ref{eqn-fsymcentraliserstructure}) were finitely generated then $\FSym$ would also be finitely generated. Hence $C_{H_n}(g)$ is not finitely generated since it has a non-finitely generated quotient.
\end{proof}
\begin{lem}\label{7.5} Fix an $r\in \N$ such that $|I_r^c(g)|= 2r$ and fix a $j\in I_r^c(g)$. Then $C_{\FSym(Y_g^{(r)})}(g)$ is a subgroup of the group generated by:
\begin{enumerate}[i)]
\item an element $h\in \Theta(g)$ with $t_j(h)\ne0$;
\item an $r$-cycle $\lambda_r:=((j, z_j(g))\;(j, z_j(g))g\ldots(j, z_j(g))g^{r-1})$; and
\item $\mu_r:=\prod_{s=1}^r((j, z_j(g))g^{s-1}\;(j, z_j(g))g^{s-1})$.
\end{enumerate}
\end{lem}
\begin{proof}
We show that the elements of $C_{\FSym(Y_g^{(r)})}(g)\cong C_r\wr\FSym\left(\quotient{Y_g^{(r)}}{g}\right)$ lie in $\langle h, \lambda_r, \mu_r\rangle$. Note that, by construction, $\Supp(h)=Y_g^{(r)}$. Thus, given any $r$-cycle $\sigma$ of $g$, there exists $d\in \Z$ such that $h^{-d}\lambda_rd^d=\sigma$ and so the base of our wreath product lies in $\langle h, \lambda_r, \mu_r\rangle$. In the space $\quotient{Y_g^{(r)}}{g}$, $h$ consists of a single infinite orbit and $\mu_r$ a single 2-cycle. It follows that $\langle h, \mu_r\rangle\cong H_2$, and so, in $\quotient{Y_g^{(r)}}{g}$, we have that $\FSym\left(\quotient{Y_g^{(r)}}{g}\right)\le \langle h, \mu_r\rangle$. Thus all finitary permutations of the orbits of $g$ of size $r$ lie in $\langle h, \mu_r\rangle$, and so $C_{\FSym(Y_g^{(r)})}(g)\le \langle h, \lambda_r, \mu_r\rangle$.
\end{proof}

\begin{lem}\label{7.6} Fix an $r\in \N$ such that $|I_r^c(g)|=sr$ where $s\ge 3$. Fix representatives $j_r^1,\ldots, j_r^s$ of $I_r^c(g)$. Then $C_{\FSym(Y_g^{(r)})}(g)$ is a subgroup of the group generated by:
\begin{enumerate}[i)]
\item elements $h_2,\ldots, h_s\in \Theta(g)$ with $t_{j_r^1}(h_i)=1$ for all $2\le i\le s$ and $t_{j_r^k}(h_k)=-1$ for each $2\le k\le s$; and
\item an $r$-cycle $\lambda_r:=((j_r^1, z_{j_r^1}(g))\;(j_r^1, z_{j_r^1}(g))g\ldots(j_r^1, z_{j_r^1}(g))g^{r-1})$.
\end{enumerate}
\end{lem}
\begin{proof}
We again show that the elements of $C_{\FSym(Y_g^{(r)})}(g)\cong C_r\wr\FSym\left(\quotient{Y_g^{(r)}}{g}\right)$ lie in $\langle h_2,\ldots, h_s, \lambda_r\rangle$. Note that
\begin{align*}
\bigcup_{i=2}^s\Supp(h_i)=Y_g^{(r)}\;\text{and that}\;\bigcap_{i=2}^s\Supp(h_i)\supseteq\{(j,m)\mid j\in [j_r^1], m\ge z_j(g)\}.
\end{align*}
Thus, given any $r$-cycle $\sigma$ of $g$, there is an $i\in \{2,\ldots,s\}$ and $d\in \Z$ such that $h_i^{-d}\lambda_rh_i^d=\sigma$, and so the base of our wreath product lies in $\langle h_2,\ldots, h_s, \lambda_r \rangle$. Now, in the space $\quotient{Y_g^{(r)}}{g}$, each $h_i$ consists of a single infinite orbit and there is a natural map $\langle h_2, \ldots h_s\rangle \twoheadrightarrow H_s$, $h_i\mapsto g_i$ where $\{g_i\mid i=2,\ldots, s\}$ denotes our standard generating set of $H_s$. To be more specific, this epimorphism is induced by the surjection of sets $X_n\rightarrow X_s$, where: for each $d\in \{2,\ldots, s\}$ and each $e\in \N$, $\{(k, z_{k}(g)+e-1)\mid k\in [j_r^d]\}\mapsto (d,e)$; for each $1\le l\le f$, $\Supp(\sigma_l)\mapsto (1,l)$ where $\sigma_1,\ldots, \sigma_f$ are the ordered $r$-cycles within $Z(g)$; and, for each $e\in \N$, $\{(k, z_k(g)+e-1)\mid k\in [j_r^1]\}\mapsto (1, f+e)$.

Thus, in $\quotient{Y_g^{(r)}}{g}$, we have that $\FSym\left(\quotient{Y_g^{(r)}}{g}\right)\le \langle h_2,\ldots, h_s\rangle$ and so all finitary permutations of the orbits of $g$ of size $r$ lie in $\langle h_2,\ldots, h_s\rangle$.\end{proof}
Note that the elements appearing in Lemma \ref{7.5} and Lemma \ref{7.6} lie in $C_{H_n}(g)$. This is because the element $\lambda_r$, of type (ii), is an orbit of $g$ of size $r$, and the elements of type (i) and (iii) have support within $Y_g^{(r)}$ and induce a permutation of the orbits of $g$ of size $r$.

\begin{defn} Let $F_a$ consist of:
\begin{enumerate}[i)]
\item the preimages of the elements $\{\Vec{\delta}_1, \ldots, \Vec{\delta}_e\}$ defined above;
\item the set $\Omega_r(a)$ (of Definition \ref{onlyfinite}) for those $r\in \N$ such that $|I_r^c(a)|=0$;
\item the elements $\lambda_r$ and $\mu_r$ of Lemma \ref{7.5} for those $r\in \N$ such that $|I_r^c(a)|=2r$;
\item the element $\lambda_r$ of Lemma \ref{7.6} for those $r\in \N$ such that $|I_r^c(a)|\ge 3r$.
\end{enumerate}
\end{defn}

\begin{proof}[Proof of Theorem 3]
By Lemma \ref{lem-actioncomputable} we may compute the set $\{g_r\ne \id\mid r \in \N\}$ and, for each $r\in \N$, the sets $I_r^c(a)$. If $|I_r^c(a)|=r$ for any $r\in \N$, then our algorithm may conclude, by Lemma \ref{7.4}, that $C_{H_n}(a)$ is not finitely generated. We now show that if this is not the case, then the set $F_a$ defined above is sufficient to generate $C_{H_n}(a)$. Note that $F_a$ is finite since $\{g_r\ne \id\mid r \in \N\}$ is finite. Also $F_a$ is computable.

Let $c\in C_{H_n}(a)$. There exist numbers $\alpha_1,\ldots,\alpha_e\in \Z$ such that $t(c)=\sum_{i=1}^e\alpha_i\Vec{\delta}_i$. Thus, by using the preimages of the elements $\{\Vec{\delta}_1, \ldots, \Vec{\delta}_e\}$ in $F_a$, we can reduce to the case where $c\in C_{\FSym(X_n)}(a)$. By Lemma \ref{thm3mainstructureresult} it is sufficient to show, for each $r\in \N$, that $C_{\FSym(Y_a^{(r)})}(a)\le \langle F_a\rangle$. For each $r\in \N$, one of the following three cases occurs: $|I_r^c(a)|=0$; $|I_r^c(a)|=2r$; or $|I_r^c(a)|\ge 3r$. In the first case, the elements of $\Omega_r(a)$ are sufficient to generate $C_{\FSym(Y_a^{(r)})}(a)$. The second and third cases were dealt with by Lemma \ref{7.5} and Lemma \ref{7.6} respectively, since $F_a$ contains the finite generating sets described in each. 
\end{proof}
\newpage
\bibliographystyle{amsalpha}

\begin{thebibliography}{BCMR14}

\bibitem[ABM15]{ConjHou}
Y.~Antol{\'{\i}}n, J.~Burillo, and A.~Martino, \emph{Conjugacy in {H}oughton's
  groups}, Publ. Mat. \textbf{59} (2015), no.~1, 3--16. \MR{3302573}
  
\bibitem[BCMR14]{Hou2}
J.~Burillo, S.~Cleary, A.~Martino, and C.~E. R{\"o}ver, \emph{Commensurations
  and {M}etric {P}roperties of {H}oughton's {G}roups}, Pacific Journal of Mathematics, \textbf{285}(2) (2016),  289--301 (doi:10.2140/pjm.2016.285.289)

\bibitem[BMV10]{orbitdecide}
O.~Bogopolski, A.~Martino, and E.~Ventura, \emph{Orbit decidability and the
  conjugacy problem for some extensions of groups}, Trans. Amer. Math. Soc.
  \textbf{362} (2010), no.~4, 2003--2036. \MR{2574885 (2011e:20045)}
\bibitem[Boo59]{Boone} W.W. Boone, \emph{The word problem}, Annals of mathematics \textbf{70}(2) (1959), 207--265.
  
\bibitem[Bro87]{brown}
K.~S. Brown, \emph{Finiteness properties of groups}, Proceedings of the
  {N}orthwestern conference on cohomology of groups ({E}vanston, {I}ll., 1985),
  vol.~44, 1987, pp.~45--75. \MR{885095 (88m:20110)}

\bibitem[CM77]{conjsub}
D.~J. Collins and C.~F. Miller, III, \emph{The conjugacy problem and
  subgroups of finite index}, Proc. London Math. Soc. (3) \textbf{34} (1977),
  no.~3, 535--556. \MR{0435227 (55 \#8187)}
  
\bibitem[Cox16]{cox2}
C. Cox, \emph{A note on the $R_\infty$ property for groups $\FAlt(X)\le G\le \Sym(X)$}, arXiv:1602.02688 (2016)

\bibitem[DM96]{Permutationgroups}
J.~D. Dixon and B. Mortimer, \emph{Permutation groups}, Graduate Texts in
  Mathematics, vol. 163, Springer-Verlag, New York, 1996. \MR{1409812
  (98m:20003)}
  
\bibitem[GS14]{rinfinity}
D. Gon{\c{c}}alves and S. Parameswaran, \emph{Sigma theory and
twisted conjugacy-II: {H}oughton groups and pure symmetric automorphism groups},
Pacific Journal of Mathematics, \textbf{280}(2) (2016),  349--369
(doi: 10.2140/pjm.2016.280.349)


\bibitem[GK75]{Gorjaga1975}
A.~V. Gorjaga and A.~S. Kirkinski{\u\i}, \emph{The decidability of the
  conjugacy problem cannot be transferred to finite extensions of groups},
  Algebra i Logika \textbf{14} (1975), no.~4, 393--406. \MR{0414718 (54
  \#2813)}

\bibitem[Hou78]{Houghton}
C.~H. Houghton, \emph{The first cohomology of a group with permutation module
  coefficients}, Archiv der Mathematik \textbf{31} (1978), 254--258.

\bibitem[Joh99]{johnson}
D.~L. Johnson, \emph{Embedding some recursively presented groups}, Groups {S}t.
  {A}ndrews 1997 in {B}ath, {II}, London Math. Soc. Lecture Note Ser., vol.
  261, Cambridge Univ. Press, Cambridge, 1999, pp.~410--416. \MR{1676637
  (2000h:20057)}

\bibitem[Laz96]{linearequations}
F. Lazebnik, \emph{On {S}ystems of {L}inear {D}iophantine {E}quations}, Math. Mag. Vol. 69, 1996, pp.~261--266

\bibitem[Lee12]{sanglee}
S.~R. Lee, \emph{Geometry of {H}oughton's {G}roups}, arXiv:1212.0257v1 (2012).

\bibitem[Mil71]{classification} C.F.~Miller III, \emph{On group-theoretic decision problems and their classification}, Annals of Math.
Studies \textbf{68}, (1971).

\bibitem[Nov58]{Novikov} P.S. Novikov \emph{On the algorithmic unsolvability of the word problem in group theory}, Trudy Mat. Inst.
Steklov. \textbf{44} (1955), 143 pages. Translation in \emph{Amer. Math. Soc. Transl.} \textbf{9}(2) (1958), 1--122.
  
\bibitem[Sco87]{grouptheory87}
W.~R. Scott, \emph{Group theory}, second ed., Dover Publications, Inc., New
  York, 1987. \MR{896269 (88d:20001)}

\bibitem[JG15]{simon}
S.~St. John-Green, \emph{Centralizers in {H}oughton's groups}, Proceedings of
  the Edinburgh Mathematical Society (Series 2) \textbf{58} (2015), 769--785.

%\bibitem[TC36]{Coxetercoset}
%J.~A. Todd and H.~S.~M. Coxeter, \emph{A practical method for enumerating
%  cosets of a finite abstract group}, Proceedings of the Edinburgh Mathematical
%  Society (Series 2) \textbf{5} (1936), 26--34.
  
\bibitem[Wie77]{ses}
J. Wiegold, \emph{Transitive groups with fixed-point-free permutations}, II. Arch. Math. (Basel) 29 (1977), no.
6, 571--573.
  
\end{thebibliography}
\providecommand{\bysame}{\leavevmode\hbox to3em{\hrulefill}\thinspace}
\providecommand{\MR}{\relax\ifhmode\unskip\space\fi MR }
\providecommand{\MRhref}[2]{%
  \href{http://www.ams.org/mathscinet-getitem?mr=#1}{#2}
}
\providecommand{\href}[2]{#2}

\end{document}